\documentclass[11pt]{amsart}

\usepackage{amsmath,amssymb,amsthm}
\usepackage{amsmath,amssymb,amsthm,amscd}
\usepackage[frame,cmtip,arrow,matrix,line,graph,curve]{xy}
\usepackage{graphpap, color}
\usepackage[mathscr]{eucal}
\usepackage{color}
\usepackage{soul}
\usepackage{accents}
\newcommand{\ubar}[1]{\underaccent{\bar}{#1}}

\def\dfrac{\displaystyle\frac}
\def\dsum{\displaystyle\sum}

\newtheorem{prop}{Proposition}
\newtheorem{theo}[prop]{Theorem}
\newtheorem{lemm}[prop]{Lemma}

\newtheorem{rema}[prop]{Remark}

\newtheorem{defi}[prop]{Definition}

\newtheorem{claim}{Claim}
\newcommand{\be}{\begin{equation}}
\newcommand{\ee}{\end{equation}}
\newcommand{\lt}{\left}
\newcommand{\rt}{\right}

\renewcommand{\leq}{\leqslant}
\renewcommand{\geq}{\geqslant}
\newcommand{\td}{\tilde}
\newcommand{\p}{\partial}

\newcommand{\la}{\kappa}
\newcommand{\R}{\mathbb{R}}
\newcommand{\M}{\mathcal{M}}
\newcommand{\K}{\mathcal{K}}

\newcommand{\us}{u^*}
\newcommand{\ur}{u^r}
\newcommand{\urs}{u^{r*}}

\newcommand{\uus}{\bar{u}^*}
\newcommand{\lus}{\ubar{u}^*}
\newcommand{\lu}{\ubar{u}}
\newcommand{\uu}{\bar{u}}
\newcommand{\ga}{\gamma}
\newcommand{\gas}{\gamma^*}
\newcommand{\vp}{\varphi}
\newcommand{\T}{\partial}

\newcommand{\w}{w^*}
\newcommand{\s}{\sigma}
\newcommand{\ka}{\kappa}
\newcommand{\goto}{\rightarrow}
\newcommand{\F}{\mathcal{F}}
\newcommand{\hF}{\hat{F}}
\newcommand{\bn}{\bar{\nabla}}
\newcommand{\C}{\mathcal{C}}

\numberwithin{equation}{section}

\begin{document}
\title{ The prescribed curvature problem for entire hypersurfaces in Minkowski space}
\author{Changyu Ren}
\address{School of Mathematical Science, Jilin University, Changchun, China}
\email{rency@jlu.edu.cn}
\author{Zhizhang Wang}
\address{School of Mathematical Science, Fudan University, Shanghai, China}
\email{zzwang@fudan.edu.cn}
\author{Ling Xiao}
\address{Department of Mathematics, University of Connecticut,
Storrs, Connecticut 06269}
\email{ling.2.xiao@uconn.edu}
\thanks{Research of the first author is supported by NSFC Grant No. 11871243 and the second author is supported  by NSFC Grant No.11871161 and 11771103.}
\maketitle

\begin{abstract} We prove three results in this paper. First, we prove for a wide class of functions $\varphi\in C^2(\mathbb{S}^{n-1})$ and
$\psi(X, \nu)\in C^2(\R^{n+1}\times\mathbb{H}^n),$ there exists a unique, entire, strictly convex, spacelike hypersurface $\M_u$ satisfying
$\s_k(\ka[\M_u])=\psi(X, \nu)$ and $u(x)\rightarrow |x|+\varphi\left(\frac{x}{|x|}\right)$ as $|x|\goto\infty.$ Second, when $k=n-1, n-2,$ we show the existence and uniqueness of entire, $k$-convex, spacelike hypersurface $\M_u$ satisfying $\s_k(\ka[\M_u])=\psi(x, u(x))$ and $u(x)\rightarrow |x|+\varphi\left(\frac{x}{|x|}\right)$ as $|x|\goto\infty.$ Last, we obtain the existence and uniqueness of entire, strictly convex, downward
translating solitons $\M_u$ with prescribed asymptotic behavior at infinity for $\s_k$ curvature
flow equations. Moreover, we prove that the  downward
translating solitons  $\M_u$ have bounded principal curvatures.
\end{abstract}

\section{Introduction}
Let $\R^{n, 1}$ be the Minkowski space with the Lorentzian metric
\[ds^2=\sum_{i=1}^{n}dx_{i}^2-dx_{n+1}^2.\]
In this paper, we will devote ourselves to the study of spacelike hypersurfaces with prescribed
$\sigma_k$ curvature in Minkowski space $\R^{n, 1}$. Here, $\s_k$ is the $k$-th elementary symmetric polynomial, i.e.,
 \[\s_k(\ka)=\sum\limits_{1\leq i_1<\cdots<i_k\leq n}\ka_{i_1}\cdots\ka_{i_k}.\]
Any such hypersurface $\M$ can be written locally as a graph of a function
$x_{n+1}=u(x), x\in\R^n,$ satisfying the spacelike condition
\be\label{int1.1}
|Du|<1.
\ee
More precisely, we will focus on the following equation:
\begin{eqnarray}\label{curvature}
\sigma_k(\ka[\M_u])=\psi(X,\nu),
\end{eqnarray}
where $X=(x, u(x))$ is the position vector of $\M_u=\{(x, u(x))| x\in\R^n\}$, $\nu=\frac{(Du, 1)}{\sqrt{1-|Du|^2}}$ is the upward unit normal lying on the hyperboloid $\mathbb{H}^n$, and $\ka[\M_u]=(\kappa_1,\cdots,\kappa_n)$ are the principal curvatures of $\M_u$. Thus equation \eqref{curvature} can be rewritten as
\begin{eqnarray}\label{curvature1}
\sigma_k(\ka[\M_u])=\psi(x,u(x),Du).
\end{eqnarray}
Notice that the right hand side functions $\psi$ of \eqref{curvature} and \eqref{curvature1} are different. Slightly extending the notation, we use the same symbol here.

 The classical Minkowski problem asks for the construction of a strictly convex compact surface $\Sigma$ whose Gaussian curvature is a given positive function $f(\nu(X))$, where $\nu(X)$ denotes the normal to $\Sigma$ at $X.$ This problem has
been discussed by Nirenberg \cite{N}, Pogorelov \cite{P3}, and Cheng-Yau \cite{CY}. The general problem of finding strictly convex hypersurfaces with prescribed surface area measures is called the Christoffel--Minkowski problem. This type of problems can be deduced to a fully nonlinear equation of the form \eqref{curvature}. It may be traced back to Alexandrov \cite{A1} who established the problem of prescribing zeroth curvature measure.
Later on, the prescribed curvature measure problem in convex geometry
has been extensively studied by Alexandrov \cite{A2}, Pogorelov
\cite{P1}, Guan-Lin-Ma \cite{GLM2}, and Guan-Li-Li \cite{GLL}. A more general form of the prescribed curvature measure problem can be expressed as
\eqref{curvature1}. In particular, Guan-Ren-Wang \cite{GRW} solved this problem in Euclidean space for convex hypersurfaces. Other related studies and references may be found in \cite{BK, CNS4, CNS5, gg, O, TW}.

Our goal here is to construct entire, spacelike hypersurfaces satisfy equation \eqref{curvature} in Minkowski space.
The main results of this paper are the following.

The first result is to construct entire, strictly convex, spacelike hypersurfaces satisfying equation \eqref{curvature}.
\begin{theo}\label{theo1}
Suppose $\varphi$ is a $C^2$ function defined on $\mathbb{S}^{n-1},$ i.e., $\varphi\in C^2(\mathbb{S}^{n-1})$, $\psi(X,\nu)\in C^2(\mathbb{R}^{n+1}\times \mathbb{H}^n)$ is a positive function, and  $c_1\geq\psi(X,\nu)\geq c_2$ for some positive constants $c_1,c_2$. We further assume that $\psi_{x_{n+1}}\geq 0$ (or $\psi_u\geq 0$). If either $\psi^{-1/k}(X, \nu)$ is locally strictly convex with respect to $X$ for any $\nu$ or $\psi$ only depends on $\nu,$ then there exits a unique, entire, strictly convex, spacelike hypersurface $\M_u=\{(x,u(x))| x\in\R^n\}$ satisfying \eqref{curvature}. Moreover, as $|x|\goto\infty,$
\be\label{approach}
u(x)\rightarrow |x|+\varphi\left(\frac{x}{|x|}\right).
\ee
\end{theo}
\begin{rema}\label{rema2}
Indeed, from the proof of the $C^2$ global estimate Lemma \ref{lem-c2-global} we can see that,
the assumption $\psi(X,\nu)$ dose not depend on $X$ can be replaced by a weaker assumption, that is,
$\psi^{-1/k}(X,\nu)$ is convex with respect to $X$ and the corresponding form $\psi(x,u,Du)$ dose not depend on $|x|$.
\end{rema}
\begin{rema}\label{rema3}
In the proof, we only can see that the hypersurface $\M_u$ we constructed is convex. In
order to say it's strictly convex, we need to apply the Constant Rank Theorem (see Theorem 1.2 in \cite{GLM} and Theorem 27 in \cite{WX})
and the Splitting Theorem (see Theorem 28 in \cite{WX}) to obtain that if $\M_u$ has a degenerate point in the interior, then
$\M_u=\M^l\times\R^{n-l},$ where $\M^l\subset\R^{l, 1}$ is a strictly convex, space like hypersurface. This contradicts \eqref{approach}.

\end{rema}
Before stating our second result, we need the following definition:
\begin{defi}
\label{intdef1}
A $C^2$ regular hypersurface $\M\subset \R^{n, 1}$ is $k$-convex,
if the principal curvatures of $\M$ at $X\in\M$ satisfy $\kappa[X]\in\Gamma_k$ for all $X\in\M$, where $\Gamma_k$
is the G\r{a}rding cone
\[\Gamma_k=\{\la\in\R^n|\sigma_m(\la)>0, m=1, \cdots, k\}.\]
\end{defi}
Using the newly developed methods in \cite{RW} and \cite{RW1}, we are able to generalize results in \cite{Bayard05}. We prove
\begin{theo}\label{theo2}
Suppose $\varphi$ is some $C^2$ function defined on $\mathbb{S}^{n-1}$ and $\psi(x,u(x))\in C^2(\mathbb{R}^{n+1})$ is a positive function satisfying $c_1\geq\psi(x, u(x))\geq c_2$ for $c_1,c_2>0$. We further assume that $k=n-1,n-2,$  and $\psi_u\geq 0$. Then there exits a unique, $k$-convex, spacelike hypersurface $\M_u=\{(x,u(x))| x\in\R^n\}$ satisfying
\be\label{curvature-k}
\s_k(\ka[\M_u])=\psi(x, u(x)).
\ee
 Moreover, as $|x|\goto\infty,$
\be\label{approach*}
u(x)\rightarrow |x|+\varphi\left(\frac{x}{|x|}\right).
\ee
\end{theo}

Now, let's consider the $\sigma_k$ curvature flow with forcing term in Minkowski space:
\begin{eqnarray}
\frac{d X}{dt}=-\left(\C-\frac{\sigma_k^{1/k}(\kappa[\M_u])}{\binom{n}{k}^{1/k}}\right)\nu,
\end{eqnarray}
where $\ka[\M_u]\in\Gamma_k.$ This can be rewritten as the equation for the height function $u,$
\begin{eqnarray}\label{flow}
\frac{u_t}{\sqrt{1-|Du|^2}}=\frac{\sigma_k^{1/k}(\kappa[\M_u])}{\binom{n}{k}^{1/k}}-\C.
\end{eqnarray}
The downward translating soliton to \eqref{flow} is of the form
\begin{eqnarray}\label{a}
u(x,t)=u(x)-t,
\end{eqnarray}
where $u(x)$ satisfies
\begin{eqnarray}\label{soliton}
\left(\frac{\sigma_k}{\binom{n}{k}}\right)^{1/k}(\ka[\M_u])=\C-\frac{1}{\sqrt{1-|Du|^2}}.
\end{eqnarray}
The above equation \eqref{soliton} can be viewed as the ``degenerate" type of \eqref{curvature}.
In this case, we prove the following theorem:
\begin{theo}\label{theo3}
Suppose $\varphi$ is a $C^2$ function defined on $\mathbb{S}^{n-1}_{\td{\C}}:=\lt\{x\in\R^n| |x|=\td{\C}\rt\},$ where $\td{\C}=\sqrt{1-\lt(\frac{1}{\C}\rt)^2}$ and $\C>1$ is a constant.
There exists a unique, strictly convex solution $u:\mathbb{R}^n\rightarrow\mathbb{R}$ of \eqref{soliton} such that as $|x|\goto\infty,$
\begin{eqnarray}\label{behaviour}
u(x)\goto\tilde{\C}|x|-\frac{1}{\C^2}\sqrt[k]{\frac{n-k}{n}}\log|x|+\varphi\left(\tilde{\C}\frac{x}{|x|}\right).
\end{eqnarray}
Moreover, $\M_u=\{(x, u(x))| x\in\R^n\}$ has bounded principal curvatures.
\end{theo}
When $k=1,$ \eqref{soliton} has been studied in \cite{JLJ} and \cite{SX16}; when $k=2,$ \eqref{soliton} has been studied in \cite{Ba2020}.

\begin{rema}
\label{uniqueness rmk}
Under our assumptions on $\psi,$ we can see that the linearized operators of equations \eqref{curvature}, \eqref{curvature-k}, and \eqref{soliton}
satisfy the maximum principle. Therefore, the uniqueness properties in Theorem \ref{theo1}, \ref{theo2}, and \ref{theo3} follow
from the maximum principle directly.
\end{rema}

The rest of this paper is organized as follows. In Section \ref{pre}, we introduce some basic
formulas and notations. The solvability of equations \eqref{curvature} and \eqref{curvature-k} on bounded domain (Dirichlet problem) is discussed in Section \ref{Dirichlet}.
We prove the local $C^1$ and $C^2$ estimates for solutions of equations \eqref{curvature} and \eqref{curvature-k} in Section \ref{local C1}. This leads to the completion of the proof of our first two main results, Theorem \ref{theo1} and Theorem \ref{theo2}, in Section \ref{prescribed curvature}. Section \ref{rs} and Section \ref{existence for soliton} are devoted to Theorem \ref{theo3}. In particular, in Section \ref{rs}, we study the radially symmetric solution to equation \eqref{soliton}, this solution will be used to construct barrier functions in Section \ref{existence for soliton}. We finish the proof of
Theorem \ref{theo3} in Section \ref{existence for soliton}.

\bigskip
\section{Preliminaries}
\label{pre}
In this paper, we will follow notations in \cite{WX}. For readers convenience, we will include some basic
notations and formulas in this section. Readers who are already familiar with calculations in Minkowski space can skip this section.

We first recall that the Minkowski space $\R^{n,1}$ is $\R^{n+1}$ endowed with the Lorentzian metric
$$ds^2=dx_1^2+\cdots dx_{n}^2-dx_{n+1}^2.$$
Throughout this paper, $\lt<\cdot, \cdot\rt>$ denotes the inner product in $\R^{n,1}$.

\subsection {Vertical graphs in $\R^{n, 1}$}
\label{vg}
A spacelike hypersurface $\M$ in $\R^{n, 1}$ is a codimension one submanifold whose induced metric
is Riemannian. Locally $\M$ can be written as a graph
\[\M_u=\{X=(x, u(x))| x\in\R^n\}\]
satisfying the spacelike condition \eqref{int1.1}. Let
$E=(0, \cdots, 0, 1),$ then the height function of $\M$ is $u(x)=-\lt<X, E\rt>.$ It's easy to see that the induced metric and second fundamental form of $\M$ are given
by
$$g_{ij}=\delta_{ij}-D_{x_i}uD_{x_j}u, \ \  1\leq i,j\leq n,$$
and
\[h_{ij}=\frac{u_{x_ix_j}}{\sqrt{1-|Du|^2}},\]
while the timelike unit normal vector field to $\M$ is
\[\nu=\frac{(Du, 1)}{\sqrt{1-|Du|^2}},\]
where $Du=(u_{x_1}, \cdots, u_{x_n})$ and $D^2u=\lt(u_{x_ix_j}\rt)$ denote the ordinary gradient and Hessian of $u$,
respectively. By a straightforward calculation, we have the principle curvatures of $\M$ are eigenvalues of the symmetric matrix
$A=(a_{ij}):$
\[a_{ij}=\frac{1}{w}\ga^{ik}u_{kl}\ga^{lj},\]
where $\ga^{ik}=\delta_{ik}+\frac{u_iu_k}{w(1+w)}$ and $w=\sqrt{1-|Du|^2}.$ Note that $(\ga^{ij})$ is invertible with inverse
$\ga_{ij}=\delta_{ij}-\frac{u_iu_j}{1+w},$ which is the square root of $(g_{ij}).$

Let $\mathcal{S}$ be the vector of $n\times n$ symmetric matrices and
\[\mathcal{S}_k=\{A\in \mathcal{S}: \lambda(A)\in \Gamma_k\},\]
where $\lambda(A)=(\lambda_1, \cdots, \lambda_n)$ denotes the eigenvalues of $A.$
Define a function $F$ by
\[F(A)=\sigma_k(\lambda(A)),\,\, A\in\mathcal{S}_k,\]
then \eqref{curvature1} can be written as
\be\label{main equation graph}
F\lt(\frac{1}{w}\ga^{ik}u_{kl}\ga^{lj}\rt)=\psi(x, u(x), Du).
\ee
Throughout this paper we denote
\[F^{ij}(A)=\frac{\p F}{\p a_{ij}}(A),\,\,F^{ij, kl}=\frac{\p^2 F}{\p a_{ij}\p a_{kl}}.\]

Now, let $\{\tau_1,\tau_2,\cdots,\tau_n\}$ be a local orthonormal frame on $T\M$. We will use $\nabla$ to denote
the induced Levi-Civita connection on $\M.$ For a function $v$ on $\M$, we denote $v_i=\nabla_{\tau_i}v,$ $v_{ij}=\nabla_{\tau_i}\nabla_{\tau_j}v,$ etc.
In particular, we have
\[|\nabla u|=\sqrt{g^{ij}u_{x_i}u_{x_j}}=\frac{|Du|}{\sqrt{1-|Du|^2}}.\]

Using normal coordinates, we also need the following well known fundamental equations for a hypersurface $\M$ in $\R^{n, 1}:$
\begin{equation}\label{Gauss}
\begin{array}{rll}
X_{ij}=& h_{ij}\nu\quad {\rm (Gauss\ formula)}\\
(\nu)_i=&h_{ij}\tau_j\quad {\rm (Weigarten\ formula)}\\
h_{ijk}=& h_{ikj}\quad {\rm (Codazzi\ equation)}\\
R_{ijkl}=&-(h_{ik}h_{jl}-h_{il}h_{jk})\quad {\rm (Gauss\ equation)},\\
\end{array}
\end{equation}
and the Ricci identity,
\begin{equation}\label{a1.2}
\begin{array}{rll}
h_{ijkl}=& h_{ijlk}+h_{mj}R_{imlk}+h_{im}R_{jmlk}\\
=& h_{klij}-(h_{mj}h_{il}-h_{ml}h_{ij})h_{mk}-(h_{mj}h_{kl}-h_{ml}h_{kj})h_{mi}.\\
\end{array}
\end{equation}
\subsection{The Gauss map}
\label{gg}
Let $\M$ be an entire, strictly convex, spacelike hypersurface, $\nu(X)$ be the timelike unit normal vector to $\M$
at $X.$ It's well known that the hyperbolic space $\mathbb{H}^{n}(-1)$ is canonically embedded in $\R^{n, 1}$
as the hypersurface
\[\lt<X, X\rt>=-1,\,\, x_{n+1}>0.\]
By parallel translating to the origin we can regard $\nu(X)$
as a point in $\mathbb{H}^n(-1).$ In this way, we define the Gauss map:
\[G: \M\rightarrow \mathbb{H}^n(-1);\,\, X\mapsto\nu(X).\]

Next, let's consider the support function of $\M.$ We denote
\[v:=\lt<X, \nu\rt>=\frac{1}{\sqrt{1-|Du|^2}}\lt(\sum_ix_i\frac{\partial u}{\partial x_i}-u\rt).\]
Let $\{e_1, \cdots, e_n\}$ be an orthonormal frame on $\mathbb{H}^n.$ We will also denote
$\{e^*_1, \cdots, e^*_n\}$ the pull-back of $e_i$ by the Gauss map $G.$ Similar to the convex geometry case,
we denote
\[\Lambda_{ij}=v_{ij}-v\delta_{ij}\]
the hyperbolic Hessian. Here $v_{ij}$ denote the covariant derivatives with respect to the hyperbolic metric.

Let $\bar{\nabla}$ be the connection of the ambient space. Then, we have
$$X=\sum_iv_ie_i-v\nu$$ and
$$\bar{\nabla}_{e_j^*}X=\sum_k(e_j(v_k)e_k+v_k\bar{\nabla}_{e_j}e_k)-v_j\nu-v\bar{\nabla}_{e_j}\nu =\sum_k\Lambda_{kj}e_k.$$
Note also that,
\begin{eqnarray}
g_{ij}&=&\lt<\bar{\nabla}_{e^*_i}X, \bar{\nabla}_{e^*_j}X\rt>=\sum_k\Lambda_{ik}\Lambda_{kj},\\
h_{ij}&=&\lt<\bar{\nabla}_{e^*_i}X, \bar{\nabla}_{e_j}\nu\rt>=\Lambda_{ij}.
 \end{eqnarray}
This implies that the eigenvalues of the hyperbolic Hessian are the curvature radius of $\M$.
Therefore, equation \eqref{curvature} can be written as
\be\label{main equation hyperbolic}
F(v_{ij}-v\delta_{ij})=\frac{1}{\psi(X, \nu)},
\ee
where $F(A)=\frac{\s_n}{\s_{n-k}}(\lambda(A)).$
Moreover, it is clear that
\be\label{gg1.1}
\left(\bar{\nabla}_{e_j}\bar{\nabla}_{e_i}\nu\right)^{\bot}=\delta_{ij}\nu,
\ee
which yields, for $k=1,2\cdots,n+1$,
\be\label{gg1.2}
 \nabla_{e_j}\nabla_{e_i}x_k=x_k\delta_{ij},
\ee
where $x_k$ is the coordinate function.

\subsection{Legendre transform}
\label{lt}
Suppose $\M$ is an entire, stictly convex, spacelike hypersurface.
Then $\M$ is the graph of a convex function
\[x_{n+1}=-\lt<X, E\rt>=u(x_1, \cdots, x_n),\]
where $E=(0, \cdots, 0, 1).$
Introduce the Legendre transform
\[\xi_i=\frac{\T u}{\T x_i},\,\, u^*=\sum x_i\xi_i-u.\]

Next, we calculate the first and the second fundamental forms in terms of $\xi_i$.
Since it is well known that,
$$\left(\frac{\T^2 u}{\T x_i\T x_j}\right)=\left(\frac{\T^2 \us}{\T \xi_i\T \xi_j}\right)^{-1}.$$
We have, the first and the second fundamental forms can be rewritten as:
$$g_{ij}=\delta_{ij}-\xi_i\xi_j, \text{ and\,\,  } h_{ij}=\frac{u^{* ij}}{\sqrt{1-|\xi|^2}},$$
where $\lt(u^{* ij}\rt)$ denotes the inverse matrix of $(\us_{ij})$ and $|\xi|^2=\sum_i\xi_i^2$. Now, let $W$ be the Weingarten matrix of $\M,$ then
$$(W^{-1})_{ij}=\sqrt{1-|\xi|^2}g_{ik}\us_{kj}.$$

From the discussion above, we can see that if $\M_u=\{(x, u(x)) | x\in\R^n\}$ is an entire, strictly convex, spacelike
hypersurface satisfying $\sigma_{k}(\la[\M])=\psi,$ then the Legendre transform of
$u$ denoted by $\us,$ satisfies
\be\label{main equation legendre}
F(\w\gas_{ik}\us_{kl}\gas_{lj})=\frac{\sigma_n}{\sigma_{n-k}}(\la^*[\w\gas_{ik}\us_{kl}\gas_{lj}])=\frac{1}{\psi}.
\ee
Here, $\w=\sqrt{1-|\xi|^2}$ and $\gas_{ij}=\delta_{ij}-\frac{\xi_i\xi_j}{1+\w}$ is the square root of the matrix $g_{ij}.$

\bigskip
\section{The Dirichlet problem}
\label{Dirichlet}
We will divide this section into two subsections. In the first subsection, we only consider the convex solution to \eqref{curvature}.
In the second subsection, we restrict ourselves to the case when $k=n-1\,(n\geq 3), n-2\,(n\geq 5),$ and we will consider the $k$-convex, spacelike solution to \eqref{curvature-k}. When $k=2,$ this problem has been studied by \cite{Ba03} and \cite{U}.
\subsection{Dirichilet problem for $1\leq k\leq n$}
\label{Dirichlet sub1}
Recall that in \cite{WX} we proved the following Lemma.
\begin{lemm}
\label{lem-legendre-boundary}
Let $\F\subset\mathbb{S}^{n-1},$ $\tilde{F}=\text{Conv}(\F),$ and $\us$ be a solution of
\be\left\{
\begin{aligned}
\hF(\w\gas_{ik}\us_{kl}\gas_{lj})&=\frac{1}{\binom{n}{k}^{\frac{1}{k}}}\,\,\text{in $\tilde{F}$}\\
\us&=\vp\,\,\text{on $\partial\tilde{F},$}
\end{aligned}
\right.
\ee
where $\hF(\w\gas_{ik}\us_{kl}\gas_{lj})=\lt(\frac{\sigma_n}{\sigma_{n-k}}\rt)^{1/k}(\ka^*[\w\gas_{ik}\us_{kl}\gas_{lj}]).$ Then, the Legendre transform of $\us$ denoted by $u$ satisfies, when $\frac{x}{|x|}\in\F$
\be\label{cv0.1}
u(x)-|x|\goto-\vp\lt(\frac{x}{|x|}\rt)\,\,\mbox{as $|x|\goto\infty,$ uniformly}.
\ee
\end{lemm}
Notice that the proof of the above Lemma is independent of the equation that the function $\us$ satisfies.
Therefore, adapting the above Lemma to the settings in this paper, this Lemma tells us that if a strictly convex function $\us: B_1\goto\R$
 satisfies $\us(\xi)=-\varphi(\xi)$ for $\xi\in\p B_1,$ then the Legendre transform of $\us$ denoted by $u,$ satisfies $u(x)\goto |x|+\varphi\lt(\frac{x}{|x|}\rt)$ as $|x|\goto\infty.$
Moreover, by Theorem 4 in \cite{WX}, there exists two solutions $\lu$, $\uu$ such that
\[\sigma_k(\ka[\M_{\lu}])=c_1,\]
\[\sigma_k(\ka[\M_{\uu}])=c_2,\]
and as $|x|\goto\infty$
\[\lu(x)-|x|,\,\,\uu(x)-|x|\goto\varphi\lt(\frac{x}{|x|}\rt).\]
Here, the constants $c_1,$ $c_2$ are the same as the ones in Theorem \ref{theo1}. Throughout this paper, we will denote the Legendre transforms of
$\lu,$ $\uu$ by $\lus,$ $\uus$ respectively. It's easy to see that $\lus$ and $\uus$ are the super- and sub- solutions of \eqref{main equation legendre}.

Combining the discussions above with Section \ref{pre}, we conclude that in order to find an entire, strictly convex solution $u$ of
\eqref{curvature1}, we only need to solve the following equation:
\be\label{Legendre curvature1}
\left\{
\begin{aligned}
F(\w\gas_{ik}\us_{kl}\gas_{lj})&=\psi^*\,\,\text{in $B_1$},\\
\us&=-\varphi\,\,\text{on $\p B_1$,}
\end{aligned}
\right.
\ee
where
$$\psi^*(\xi,u^*,Du^*)=\frac{1}{\psi(x,u,Du)}=\frac{1}{\psi(Du^*,\xi\cdot Du^*-u^*,\xi)},$$ and
$$F(\w\gas_{ik}\us_{kl}\gas_{lj})=\frac{\sigma_n}{\sigma_{n-k}}(\la^*[\w\gas_{ik}\us_{kl}\gas_{lj}]).$$
Note that by our assumption in Theorem \ref{theo1} we have,
\begin{eqnarray}
\psi^*_{u^*}=\frac{\psi_u}{\psi^2}\geq 0.
\end{eqnarray}
Thus, equation \eqref{Legendre curvature1} possesses the maximum principle.

Notice that equation \eqref{Legendre curvature1} is degenerate on $\p B_1$. Therefore, we will consider the approximate equation:
\be\label{D-approximate}
\left\{
\begin{aligned}
F(\w\gas_{ik}\us_{kl}\gas_{lj})&=\psi^*\,\,\text{in $B_r$},\\
\us&=\lu^*\,\,\text{on $\p B_r$}
\end{aligned}
\right.
\ee
where $0<r<1.$

By continuity method we know that, if we can obtain a prior estimates up to the second order, then we can show \eqref{D-approximate} has a unique, strictly convex solution $\urs.$ In view of the super- and sub- solutions $\lu^*,\uu^*$, the $C^0$ estimates are easy to obtain. The $C^1$ estimates can be derived by following the argument in Subsection 9.2 of \cite{RWX}. The $C^2$  estimate on the boundary can be derived from  Lemma 27 in \cite{RWX} and the argument of Bo Guan \cite{Guan}. In the following, we only need to consider the global $C^2$ estimate.

Let $\M_u=\{(x, u(x))| x\in\R^n\}$ be a strictly convex, spacelike hypersurface, $v=\lt<X, \nu\rt>$ be the support function of $\M_u,$ and $\us$ be the Legendre transform of $u.$  From Subsection \ref{gg} and \ref{lt}, we know that $\lambda[v_{ij}-v\delta_{ij}]=\ka^*[\w\gas_{ik}\us_{kl}\gas_{lj}].$ Therefore, to study the global $C^2$ estimate of \eqref{D-approximate} is equivalent to study the global $C^2$ estimate of \eqref{main equation hyperbolic}.

For our convenience,
we will consider the equation
\be
\label{D6}
\hat{F}(\Lambda)=\left(\frac{\sigma_n}{\sigma_{n-k}}\right)^{\frac{1}{k}}(\Lambda)=\tilde{\psi},
\ee
where $\Lambda=(\Lambda_{ij})=(v_{ij}-v\delta_{ij}),$ $\tilde{\psi}=\psi^{-1/k}(X,\nu),$ and $v_{ij}$ is the covariant derivatives with respect to the hyperbolic metric.

We will use $\lambda[\Lambda]=(\lambda_1,\lambda_2,\cdots,\lambda_n)$ to denote the eigenvalues of the matrix $\Lambda$.
We define the Riemann curvature tensor:
$$R(X,Y)=\nabla_X\nabla_Y-\nabla_Y\nabla_X-\nabla_{[X,Y]}.$$
Let $\{e_1,e_2,\cdots,e_n \}$ be an orthonormal frame on $\mathbb{H}^n,$ we use the notation
$$R_{ijkl}=R(e_i,e_j)e_k\cdot e_l,\,\, R^l_{ijk}=g^{lp}R_{ijkp}.$$ Then the commutation formulae are
$$v_{ijk}-v_{ikj}=R^l_{jki}v_l,\,\, v_{ijkl}-v_{ijlk}=R_{kli}^mv_{jm}+R^{m}_{klj}v_{im}.$$
Note that in hyperbolic space we have,
$$R_{ijkl}=g_{ik}g_{jl}-g_{il}g_{jk}.$$
Therefore, given an orthonormal frame on $\mathbb{H}^n,$ we obtain the following geometric formulae:
\begin{eqnarray}\label{comm}
\Lambda_{ijk}&=&\Lambda_{ikj}\\
\Lambda_{lkji}-\Lambda_{lkij}&=&v_{lkji}-v_{lkij}\nonumber\\
&=&-v_{lj}\delta_{ik}+v_{li}\delta_{jk}-v_{jk}\delta_{il}+v_{ik}\delta_{jl}\nonumber.
\end{eqnarray}
We will prove
\begin{lemm}
\label{lem-c2-global}
Let $v$ be the solution of \eqref{D6} in a bounded domain $U\subset\mathbb{H}^n$. Denote the eigenvalues of $(v_{ij}-v\delta_{ij})$ by $\lambda[v_{ij}-v\delta_{ij}] =(\lambda_1, \cdots, \lambda_n).$ Then
\[\lambda_{\max}\leq \max\{C, \lambda|_{\p U}\},\]
where $\lambda_{\max}=\max\{\lambda_1, \cdots, \lambda_n\},$ and $C$ is a positive constant only depending on  $U$ and $\tilde{\psi}.$
\end{lemm}
\begin{proof}Set
$$M=\max\limits_{P\in \overline{U}}\max\limits_{|\xi|=1, \xi\in T_{P}\mathbb{H}^n}\left(\log \Lambda_{\xi\xi}+Nx_{n+1}\right),$$
where $x_{n+1}$ is the coordinate function.
Without loss of generality, we assume $M$ is achieved at an interior point $P_0\in U$ for some direction $\xi_0.$ Chose an orthonormal frame $\{e_1, \cdots, e_n\}$ around $P_0$ such that $e_1(P_0)=\xi_0$ and $\Lambda_{ij}(P_0)=\lambda_i\delta_{ij}.$

Now, let's consider the test function
$$\phi=\log \Lambda_{11}+Nx_{n+1}.$$
At its maximum point $P_0$, we have
\begin{eqnarray}\label{3.6}
0=\phi_i&=&\frac{\Lambda_{11i}}{\Lambda_{11}}+N(x_{n+1})_i\\
0\geq \phi_{ii}&=&\frac{\Lambda_{11ii}}{\Lambda_{11}}-\frac{\Lambda_{11i}^2}{\Lambda_{11}^2}+N(x_{n+1})_{ii}.
\end{eqnarray}
Note that $(x_{n+1})_{ij}=x_{n+1}\delta_{ij}$, thus
\begin{eqnarray}\label{3.8}
\hF^{ii}\phi_{ii}=\frac{\hF^{ii}\Lambda_{11ii}}{\Lambda_{11}}-\frac{\hF^{ii}\Lambda_{11i}^2}{\Lambda_{11}^2}+Nx_{n+1}\sum_i \hF^{ii}.
\end{eqnarray}
In view of \eqref{comm}, we get
\begin{eqnarray}\label{3.9}
\Lambda_{11ii}=\Lambda_{i11i}=\Lambda_{i1i1}+v_{ii}-v_{11}=\Lambda_{ii11}+\Lambda_{ii}-\Lambda_{11}.\nonumber
\end{eqnarray}
This yields,
\begin{eqnarray}\label{3.10}
\hF^{ii}\Lambda_{11ii}=\hF^{ii}\Lambda_{ii11}+\hF^{ii}\Lambda_{ii}-\Lambda_{11}\sum_{i}\hF^{ii}.
\end{eqnarray}
Differentiating equation \eqref{D6} twice we obtain,
\begin{eqnarray}\label{3.11}
\hF^{ii}\Lambda_{ii11}&=&-\hF^{pq,rs}\Lambda_{pq1}\Lambda_{rs1}+\tilde{\psi}_{11}\\
&=&-\hF^{pp,qq}\Lambda_{pp1}\Lambda_{qq1}-\sum_{p\neq q}\frac{\hF^{pp}-\hF^{qq}}{\lambda_p-\lambda_q}\Lambda_{pq1}^2+\tilde{\psi}_{11}.\nonumber
\end{eqnarray}
By the concavity of $(\sigma_n/\sigma_{n-k})^{1/k}$ we can see that the first term on the right hand side is nonnegative.
Combining \eqref{3.8}-\eqref{3.11} we have,
\begin{eqnarray}\label{3.12}
\\
\hF^{ii}\phi_{ii}&\geq &\frac{\tilde{\psi}_{11}}{\Lambda_{11}}-\frac{1}{\Lambda_{11}}\sum_{p\neq q}\frac{\hF^{pp}-\hF^{qq}}{\lambda_p-\lambda_q}\Lambda_{pq1}^2-\frac{\hF^{ii}\Lambda_{11i}^2}{\Lambda_{11}^2}+(Nx_{n+1}-1)\sum_i \hF^{ii}\nonumber\\
&\geq&\frac{\tilde{\psi}_{11}}{\Lambda_{11}}+\frac{1}{\Lambda_{11}}\sum_{i\neq 1}\frac{\hF^{ii}-\hF^{11}}{\lambda_1-\lambda_i}\Lambda_{11i}^2-\frac{\hF^{ii}\Lambda_{11i}^2}{\Lambda_{11}^2}+(Nx_{n+1}-1)\sum_i \hF^{ii}\nonumber
\end{eqnarray}
We need an explicit expression of $\hF^{ii}$. A straightforward calculation gives
\begin{eqnarray}\label{FF}
k\hF^{k-1}\hF^{ii}=\frac{\sigma_n^{ii}\sigma_{n-k}-\sigma_n\sigma_{n-k}^{ii}}{\sigma_{n-k}^2},
\end{eqnarray}
where for $1\leq l\leq n,$ $\sigma^{ii}_l=\frac{\p\s_l}{\p\lambda_i}.$
Since
\begin{eqnarray}
&&\sigma_n^{ii}\sigma_{n-k}-\sigma_n\sigma_{n-k}^{ii}\nonumber\\
&=&\sigma_{n-1}(\lambda|i)(\lambda_i\sigma_{n-k-1}(\lambda|i)+\sigma_{n-k}(\lambda|i))
-\lambda_i\sigma_{n-1}(\lambda|i)\sigma_{n-k-1}(\lambda|i)\nonumber\\
&=&\sigma_{n-1}(\lambda|i)\sigma_{n-k}(\lambda|i)\nonumber.
\end{eqnarray}
Here and in the following, $\s_l(\lambda|a)$ and $\s_l(\lambda|ab)$ are the $l$-th elementary symmetric polynomials of
$\lambda_1, \cdots, \lambda_n$ with $\lambda_a=0$ and $\lambda_a=\lambda_b=0,$ respectively.
It follows
\begin{eqnarray}\label{F}
k\hF^{k-1}\hF^{ii}=\frac{\sigma_{n-1}(\lambda|i)\sigma_{n-k}(\lambda|i)}{\sigma_{n-k}^2}.
\end{eqnarray}
Therefore, we get
\begin{eqnarray}
&&k\hF^{k-1}(\hF^{ii}-\hF^{11})\\
&=&\frac{1}{\sigma_{n-k}^2}[\sigma_{n-1}(\lambda|i)\sigma_{n-k}(\lambda|i)-\sigma_{n-1}(\lambda|1)\sigma_{n-k}(\lambda|1)]\nonumber\\
&=&\frac{\sigma_{n-2}(\lambda|1i)}{\sigma_{n-k}^2}[\lambda_1\sigma_{n-k}(\lambda|i)-\lambda_i\sigma_{n-k}(\lambda|1)]\nonumber\\
&=&\frac{\sigma_{n-2}(\lambda|1i)(\lambda_1-\lambda_i)}{\sigma_{n-k}^2}[(\lambda_1+\lambda_i)\sigma_{n-k-1}(\lambda|1i)+\sigma_{n-k}(\lambda|1i)]\nonumber.
\end{eqnarray}
When $i\geq 2$, we can see that
\begin{eqnarray}\label{3.14}
&&k\hF^{k-1}\left(\frac{\hF^{ii}-\hF^{11}}{\lambda_1-\lambda_i}-\frac{\hF^{ii}}{\lambda_1}\right)\\
&=&\frac{\sigma_{n-2}(\lambda|1i)}{\sigma_{n-k}^2}[(\lambda_1+\lambda_i)\sigma_{n-k-1}(\lambda|1i)+\sigma_{n-k}(\lambda|1i)-\sigma_{n-k}(\lambda|i)]\nonumber\\
&=&\frac{\sigma_{n-2}(\lambda|1i)}{\sigma_{n-k}^2}\lambda_i\sigma_{n-k-1}(\lambda|1i)\nonumber\\
&=&\frac{\sigma_{n-1}(\lambda|1)}{\sigma_{n-k}^2}\sigma_{n-k-1}(\lambda|1i)\nonumber\\
&>&0.\nonumber
\end{eqnarray}
Plugging \eqref{3.14} into \eqref{3.12}, we obtain
\begin{eqnarray}\label{3.17}
\hF^{ii}\phi_{ii}&\geq& \frac{\tilde{\psi}_{11}}{\Lambda_{11}}-\hF^{11}\frac{\Lambda_{11i}^2}{\Lambda_{11}^2}+(Nx_{n+1}-1)\sum_i\hF^{ii}\\
&=&\frac{\tilde{\psi}_{11}}{\Lambda_{11}}-\hF^{11}N^2(y_{n+1})_1^2+(Nx_{n+1}-1)\sum_i\hF^{ii}.\nonumber
\end{eqnarray}
Here, in the last equality, we have used \eqref{3.6}.

Now, let's calculate $\tilde{\psi}_{11}$.  We denote the connection of the ambient space by $\bn$, and $\{e_1^*,e_2^*,\cdots,e^*_n\}$ denotes the pull back of $\{e_1,e_2,\cdots,e_n\}$ via the Gauss map.
Differentiating $\td{\psi}$ with respect to $e_1$ twice we get,
\be
\tilde{\psi}_{1}=d_X\psi^{-1/k}(\bn_{e^*_1}X)+d_{\nu}\psi^{-1/k}(e_1),
\ee
and
\begin{eqnarray}\label{D1.0}
\tilde{\psi}_{11}&=&d_Xd_X\psi^{-1/k}(\bn_{e^*_1}X,\bn_{e^*_1}X)+d_X\psi^{-1/k}(\bn_{e_1}\bn_{e^*_1}X)\\
&+&2d_Xd_{\nu}\psi^{-1/k}(e_1, \bn_{e^*_1}X)+d_{\nu}d_{\nu}\psi^{-1/k}(e_1,e_1)+d_{\nu}\psi^{-1/k}(\bn_{e_1}e_1)\nonumber\\
&\geq&c_0\Lambda_{11}^2+d_X\psi^{-1/k}(\bn_{e_1}\sum_k\Lambda_{k1}e_k)+2d_Xd_{\nu}\psi^{-1/k}(e_1, \sum_l\Lambda_{l1}e_l)\nonumber\\
&+&d_{\nu}d_{\nu}\psi^{-1/k}(e_1,e_1)+d_{\nu}\psi^{-1/k}(\nu)\nonumber\\
&\geq&c_0\Lambda_{11}^2+\sum_kd_X\psi^{-1/k}(\Lambda_{k11}e_k+\Lambda_{k1}\delta_{k1}\nu)-C\lambda_{1}-C\nonumber\\
&\geq&c_0\Lambda_{11}^2+\sum_k\Lambda_{11k}d_X\psi^{-1/k}(e_k)-C\lambda_1-C\nonumber,
\end{eqnarray}
where the first inequality comes from the locally strict convexity assumption on $\psi^{-1/k},$ i.e., for any spacelike vector $\xi\in\mathbb{R}^{n,1}$,
$$d_Xd_X\psi^{-1/k}(\xi,\xi)\geq c_0|\xi|^2_{E}\geq c_0|\xi|^2_{M}.$$ Here $c_0>0$ is some constant depending on the defining domain, and $|\cdot|_E,|\cdot|_M$ are the Euclidean norm and Minkowski norm respectively.
At the point $P_0$, in view of \eqref{3.6} and the assumption that $\psi_{x_{n+1}}\geq 0$ , we derive
\be\label{3.19}
\begin{aligned}
\frac{\tilde{\psi}_{11}}{\Lambda_{11}}&\geq c_0\lambda_1-N\sum_k(x_{n+1})_kd_X\psi^{-1/k}(e_k)-C-\frac{C}{\lambda_1}\\
&=c_0\lambda_1+\frac{N}{k}\psi^{-1/k-1}d_X\psi(\nabla x_{n+1})-C-\frac{C}{\lambda_1}\\
&=c_0\lambda_1+\frac{N}{k}\psi^{-1/k-1}d_X\psi\left(-\frac{\p}{\p x_{n+1}}+x_{n+1}\nu\right)-C-\frac{C}{\lambda_1}\\
&=c_0\lambda_1+\frac{N}{k}\psi^{-1/k-1}d_X\psi\left(|x|^2\frac{\p}{\p x_{n+1}}+x_{n+1}\sum_{i=1}^n x_{i}\frac{\p}{\p x_i}\right)-C-\frac{C}{\lambda_1}\\
&=c_0\lambda_1+\frac{N|x|^2}{k}\psi^{-1/k-1}\frac{\p \psi}{\p x_{n+1}}+\frac{N}{k}\psi^{-1/k-1}x_{n+1}\sum_{i=1}^nx_i\frac{\p \psi}{\p x_{i}}-C-\frac{C}{\lambda_1}\\
&\geq c_0\lambda_1+\frac{N}{k}\psi^{-1/k-1}x_{n+1}\sum_{i=1}^nx_i\frac{\p \psi}{\p x_{i}}-C-\frac{C}{\lambda_1}\\
&\geq-C-\frac{C}{\lambda_1}.
\end{aligned}
\ee
Here, in the last inequality we have assumed $\lambda_1=\lambda_1(|\psi|_{C^2})>0$ is large at $P_0$.
On the other hand, note that the functional $\hF$ is concave and homogenous of degree one. Therefore,
\be\label{3.21}
\begin{aligned}
\sum_{i}\hF^{ii}&=\hF(\lambda)+\sum_{i}\hF^{ii}(1-\lambda_i)\\
&\geq\hF(1)=\binom{n}{k}^{-1/k}.
\end{aligned}
\ee
Combining \eqref{3.17}-\eqref{3.21}, we obtain
\begin{eqnarray}
0\geq \hF^{ii}\phi_{ii}&\geq& -C-\frac{C}{\lambda_1}-\frac{C}{\lambda_1}N^2(x_{n+1})_1^2+(Nx_{n+1}-1)\binom{n}{k}^{-1/k}.\nonumber
\end{eqnarray}
Let $N, \lambda_1$ be sufficiently large, then we obtain a contradiction. This completes the proof of Lemma \ref{lem-c2-global}.

Notice that this is the only place we need to use the locally strict convexity assumption of $\psi^{-1/k}$ in
Theorem \ref{theo1}. It's also clear that the above proof can be easily modified to the case when $\psi^{-1/k}$ is convex with respect to $X$ and the corresponding $\psi(x,u(x),Du)$ does not depend on
$|x|$ (see the second inequality in \eqref{3.19}), as stated in the Remark \ref{rema2}. Therefore, \eqref{D-approximate} is solvable when either $\psi^{-1/k}$ is locally strictly convex with respect to $X$ or $\psi^{-1/k}$ is convex with respect to $X$ and $\psi(x, u(x), Du(x))$ does not depend on $|x|.$
\end{proof}
\par
\subsection{Dirichilet problem for $k=n-1, n-2$}
\label{Dirichlet sub2}
Let $n\in\mathbb{N}$ and $\Omega_n:=\{x\in\R^n|\lu(x)=n\},$  we will consider the following
Dirichlet problem:
\be\label{Dirichlet curvature}
\left\{
\begin{aligned}
\s_k(\ka[\M_u])&=\psi(x, u(x))\,\,\text{in $\Omega_n$},\\
u&=n\,\,\text{on $\p \Omega_n$.}
\end{aligned}
\right.
\ee
Note that since $\lu$ is strictly convex, $\Omega_n$ is strictly convex.
It's easy to see that if $u$ is a solution of \eqref{Dirichlet curvature}, then $\lu\leq u\leq \uu.$ Therefore, in order to find a $k$-convex solution $u$ for
\eqref{Dirichlet curvature}, we only need to study the
$C^1$ and $C^2$ estimates of $u.$
\subsubsection{$C^1$ estimate for equation \eqref{Dirichlet curvature}}
\begin{lemm}
\label{lem gradient}
Let $u$ be a solution of \eqref{Dirichlet curvature}, then $|Du|<C<1.$ Here $C$ is a constant depending on $|D\lu|_{\bar{\Omega}_n}$ and $\psi.$
\end{lemm}
\begin{proof}
Let $V=-\lt<\nu, E\rt>=\frac{1}{\sqrt{1-|Du|^2}},$ and consider the test function $\phi=\ln V+Ku,$ where $K>0$ to be determined. If $\phi$ achieves its maximum at an interior point
$P_0\in\M_u,$ then at this point, we may choose a normal coordinate $\{\tau_1, \cdots, \tau_n\}$ such that $h_{ij}=\ka_i\delta_{ij}.$ Since at $P_0$ we have
\[\phi_i=\frac{V_i}{V}+Ku_i=0\]
and \[0\geq\phi_{ii}=\frac{V_{ii}}{V}-\frac{V_i^2}{V^2}+Ku_{ii}.\]
A straightforward calculation yields
\[0\geq-\frac{\lt<\nabla\s_k, E\rt>}{V}-\frac{\s_k^{ii}\ka_i^2u_i^2}{V^2}+Kk\psi V+\s_k^{ii}\ka_i^2.\]
Note that $\lt|\lt<\nabla\s_k, E\rt>\rt|\leq CV^2,$ where $C$ only depends on $|\psi|_{C^1}.$ Choose $K>C+1$ we have
$$-\frac{\lt<\nabla\s_k, E\rt>}{V}-\frac{\s_k^{ii}\ka_i^2u_i^2}{V^2}+Kk\psi V+\s_k^{ii}\ka_i^2>0.$$ This leads to a contradiction.
\end{proof}
\par
\subsubsection{$C^2$ boundary estimates for equation \eqref{Dirichlet curvature}}
Now, we will establish the $C^2$ boundary estimate. For our convenience, we will consider the solvability of the following Dirichlet problem:
\be\label{Dirichlet curvature*}
\left\{
\begin{aligned}
G(Du, D^2u)=F\lt(\frac{1}{w}\ga^{ik}u_{kl}\ga^{lj}\rt)&=\psi(x, u(x))\,\,\text{in $\Omega$},\\
u&=0\,\,\text{on $\p \Omega$,}
\end{aligned}
\right.
\ee
where $\Omega$ is strictly convex.
We will follow the idea of \cite{CNS5}.

\textbf{Infinitesmal stretching.} If $u$ is a solution of \eqref{Dirichlet curvature*}, let $v(x)=\frac{1}{t}u(tx),$ where $t>0.$ Then the principal curvatures
of $\M_v$ satisfies $\ka[\M_v(x)]=t\ka[\M_u(tx)].$  Therefore
\be\label{inf1}
\begin{aligned}
G(Dv, D^2v)&=t^k\psi(tx, u(tx))\\
&=t^k\psi(tx, tv(x)).
\end{aligned}
\ee
We denote $\dot{v}=\frac{d}{dt}v=-\frac{1}{t^2}u(tx)+x\cdot Du(tx),$ when $t=1$
$$\dot{v}=x\cdot Du(x)-u(x).$$
Differentiating equation \eqref{inf1} with respect to $t,$ then evaluate it at $t=1$ we obtain
\begin{align*}
&G^{ij}\p_{ij}\dot{v}+G^s\p_s\dot{v}\\
=&k\psi+\psi_z(v+\dot{v})+x\psi_x.\\
\end{align*}
Denote $L:=G^{ij}\p_{ij}+G^s\p_s,$ we have
\be\label{inf2}
\begin{aligned}
L(x\cdot Du-u)&=k\psi+\psi_{z}(u+x\cdot Du-u)+x\psi_x\\
&=k\psi+x\psi_x+\psi_zx\cdot Du.
\end{aligned}
\ee

\textbf{Infinitesmal rotation in Minkowski space.} Keeping the coordinates $x'=(x_1, \cdots, x_{n-1})$ fixed, we rotate in the
$(x_n, u)$ variables,
$$\begin{bmatrix}
\cosh\theta & \sinh\theta\\
\sinh\theta &\cosh\theta
\end{bmatrix}\begin{bmatrix}x_n\\u\end{bmatrix}=\begin{bmatrix}\cosh\theta x_n+\sinh\theta u\\
\cosh\theta u+\sinh\theta x_n\end{bmatrix}.$$
To the first order in $\theta$ the image of $(x, u(x))$ under such rotation is
\[(x', x_n+u(x)\theta, u(x)+x_n\theta).\]
Therefore, to the first order in $\theta$ the image of
\[(x', x_n-u(x)\theta, u(x', x_n-u(x)\theta))\]
is $(x', x_n, u(x', x_n-u(x)\theta)+x_n\theta).$ Denote this image as a graph function
$$v(x)=u(x', x_n-u(x)\theta)+x_n\theta+\mbox{higher order in $\theta$},$$ then we have
\begin{align*}
G(Dv, D^2v)&=\psi(x', x_n-u(x)\theta, u(x', x_n-u(x)\theta))+\mbox{higher order in $\theta$}\\
&=\psi(x', x_n-u(x)\theta, v(x)-x_n\theta)+\mbox{higher order in $\theta$}.
\end{align*}
Notice that $\left.\frac{dv}{d\theta}\right|_{\theta=0}=x_n-u_nu,$ we obtain
\be\label{inf3}
\begin{aligned}
&G^{ij}\p_{ij}(x_n-u_nu)+G^s\p_s(x_n-u_nu)\\
&=\psi_n(-u(x))+\psi_z(x_n-u_nu-x_n).
\end{aligned}
\ee
Thus, we conclude that
\be\label{inf4}
L(x_n-uu_n)=-u\psi_n-u_nu\psi_z.
\ee

\begin{lemm}
\label{lem c2 boundary}
Let $u$ be a solution of \eqref{Dirichlet curvature*}, then $|D^2u|<C$ on $\p \Omega.$ Here $C$ is a constant depending on $\Omega$ and $\psi.$
\end{lemm}
\begin{proof}
For any $p\in\p\Omega,$ we suppose $p$ is the origin and that the $x_n-$ axis is the interior normal of $\p\Omega$ at $p.$
We may also assume the boundary near the origin $p$ is represented by
\[x_n=\frac{1}{2}\sum\limits_{\alpha=1}^{n-1}\lambda_\alpha x_{\alpha}^2+O(|x'|^3),\,\,x'=(x_1, \cdots, x_{n-1}),\]
where $\lambda_\alpha>0,\,\,1\leq\alpha\leq n-1$ are the principal curvatures of $\p\Omega$ at the origin.
Let $T_\alpha=\p_\alpha+\lambda_\alpha(x_\alpha\p_n-x_n\p_\alpha).$
Note that $G^{ij}u_{ij\alpha}+G^su_{s\alpha}=\psi_\alpha+\psi_zu_\alpha.$
In view of the fact that \eqref{Dirichlet curvature} is invariant under rotation ( see equation (3.1) in \cite{CNS5}), we get
\be\label{inf5}
|LT_\alpha u|\leq C.
\ee
Moreover, it's easy to see we have $|T_\alpha u|\leq C|x'|^2$ on $\p\Omega$ near the origin.
In the following, we denote $\Omega_\beta:=\Omega\cap\{x_n<\beta\}.$ Set
\[h=(x\cdot Du-u)-\frac{\delta}{\beta}(x_n-uu_n).\]
On $\p\Omega\cap\p\Omega_\beta,$ note that $u=0,$ we have $x\cdot Du\leq C_1|x'|^2.$ This implies
on $\p\Omega\cap\p\Omega_\beta,$
\be\label{inf6}
h=x\cdot Du-\frac{\delta}{\beta}x_n\leq\lt(C_1-\frac{\delta}{\beta}a\rt)|x'|^2,
\ee
where $a>0$ depends on the principal curvatures of $\p\Omega.$
Notice that $u$ is a spacelike function, we suppose $|Du|\leq\theta_0$ in $\bar{\Omega}$ for some
$\theta_0\in (0, 1).$ Then we have $0\leq-u\leq\theta_0\beta$ in $\Omega_\beta.$ Therefore, on $\{x_n=\beta\}$
we obtain
\be\label{inf7}
\begin{aligned}
h&=\beta u_n+\sum\limits_{\alpha=1}^{n-1}x_\alpha u_\alpha-u+\frac{\delta}{\beta}uu_n-\delta\\
&\leq\beta\theta_0+C\beta^{1/2}+\theta_0\beta+\theta_0^2\delta-\delta\\
&\leq C\beta^{1/2}+\delta(\theta_0-1)
\end{aligned}
\ee
with $C$ being independent of $\beta$ and $\delta.$
Moreover,
\be\label{inf8}
\begin{aligned}
Lh&=k\psi+x\psi_x+\psi_zx\cdot Du-\frac{\delta}{\beta}(-u\psi_n-u_nu\psi_z)\\
&\geq k\psi-C\beta^{1/2}-C\delta\\
&\geq\frac{k}{2}\psi,
\end{aligned}
\ee
where $\delta$ and $\beta$ are small positive constants.

Now choose $A=A(\delta)>0$ large such that
\[Ah\leq-|T_\alpha u|\,\,\mbox{on $\p\Omega_\beta,$}\]
and $LAh>|LT_{\alpha}u|$ in $\Omega_\beta.$
By the maximum principle we conclude that
\[Ah\pm T_\alpha u\leq 0\,\,\mbox{in $\bar{\Omega}_\beta.$}\]
On the other hand we have $h(0)=T_\alpha u(0)=0.$ Therefore,
\[|\p_nT_\alpha u(0)|\leq -Ah_n(0)\leq\frac{A\delta}{\beta},\]
which yields
\[|u_{n\alpha}(0)|\leq C.\]
Since $p\in\p\Omega$ is arbitrary, we get
\[|u_{\alpha n}(x)|\leq C\,\,\mbox{for any $x\in\p\Omega.$}\]
Applying Lemma 1.2 in \cite{CNS3} we obtain
\[|u_{nn}(x)|\leq C\,\,\mbox{for any $x\in\p\Omega.$}\]
This completes the proof of this Lemma.
\end{proof}

\par
\subsubsection{$C^2$ global estimate for equation \eqref{Dirichlet curvature}}
Finally, we will prove the $C^2$ global estimate. In this subsubsection, for the greater generality, we will assume
$\psi=\psi(X, \nu).$
\begin{lemm}
\label{lem c2 global}
Let $u$ be a solution of \eqref{Dirichlet curvature*} with $\psi=\psi(X, \nu)$, then 
$$|D^2u|<\max\{C, \max\limits_{\p \Omega}|D^2u|\}$$ on $\Omega.$ Here $C$
is a constant depending on $|Du|_{\Omega}$ and $\psi.$
\end{lemm}
\begin{proof}
we consider the following test function whose form first appeared in
\cite{GRW},
$$
\phi=\log\log P-N\lt<\nu, E\rt>.
$$
Here, the function $P$ is defined by $P=\dsum_le^{\kappa_l}$ and $N$ is a sufficiently large constant to be determined later.

We may assume that the maximum of $\phi$ is achieved  at some point
$P_0\in \M_u$, where $u$ is the solution of \eqref{Dirichlet curvature*}. Suppose $\{\tau_1,\tau_2,\cdots,\tau_n\}$ is a normal coordinate near $P_0$
such that at $P_0,$ $h_{ij}=\la_i\delta_{ij}$ and $\la_1\geq\la_2\geq\cdots\geq\la_n.$
\par
Differentiating the function $\phi$ twice at $P_0$, we have
\begin{equation}\label{dp2.11}
\phi_i =\dfrac{P_i}{P\log P}+Nh_{ii}u_i = 0,
\end{equation}
and
\begin{eqnarray*}
&&\phi_{ii}\\
&=& \frac{P_{ii}}{P\log P}-\frac{P_i^2}{P^2\log P}-\frac{P_i^2}{(P\log P)^2}-Nh_{ii}^2\lt<\nu, E\rt>+\sum_sNu_sh_{isi}\nonumber\\
&=&\frac{1}{P\log
P}\bigg[\sum_le^{\kappa_l}h_{llii}+\sum_le^{\kappa_l}h_{lli}^2+\sum_{p\neq
q}\frac{e^{\kappa_{p}}
-e^{\kappa_{q}}}{\kappa_{p}-\kappa_{q}}h_{pq i}^2-\Big(\frac{1}{P}+\frac{1}{P\log P}\Big)P_i^2\bigg]\nonumber\\
&&-Nh_{ii}^2\lt<\nu, E\rt>+\sum_sNu_sh_{iis}\nonumber
\end{eqnarray*}
Contracting with $\sigma_{k}^{ii}$, we get
\begin{eqnarray}\label{dp2.13}
&&\sigma_{k}^{ii}\phi_{ii}\\
&=&\frac{\sigma_{k}^{ii}}{P\log
P}\bigg[\sum_le^{\kappa_l}h_{llii}+\sum_le^{\kappa_l}h_{lli}^2+\sum_{p\neq
q}\frac{e^{\kappa_{p}}
-e^{\kappa_{q}}}{\kappa_{p}-\kappa_{q}}h_{pq i}^2-\Big(\frac{1}{P}+\frac{1}{P\log P}\Big)P_i^2\bigg]\nonumber\\
&&-N\sigma_{k}^{ii}\kappa_{i}^2\lt<\nu, E\rt>+\sum_sNu_s\sigma_{k}^{ii}h_{iis}.\nonumber
\end{eqnarray}\par
At $P_0$, differentiating the equation (\ref{curvature}) twice yields,
\begin{align}\label{dp2.14}
\sigma_{k}^{ii}h_{iil}=d_X\psi(\tau_l)+\kappa_ld_{\nu}\psi(\tau_l),
\end{align}
and
\begin{align}\label{dp2.15}
\sigma_{k}^{ii}h_{iill}+\sigma_{k}^{pq,rs}h_{pql}h_{rsl}\geq
-C-Ch_{11}^2+\sum_sh_{sll}d_{\nu}\psi(\tau_s),
\end{align}
where $C$ is some uniform constant only depending on $\psi$. Note that
\begin{eqnarray}\label{newdp216}
h_{llii}=h_{iill}-h_{ii}h_{ll}^2+h^2_{ii}h_{ll}.
\end{eqnarray}
Inserting  (\ref{dp2.15}) and \eqref{newdp216} into (\ref{dp2.13}), we obtain
\begin{eqnarray}
&&\sigma_{k}^{ii}\phi_{ii}\label{dp2.16}\\
&\geq &\frac{1}{P\log P}\bigg[\sum_le^{\kappa_l}\Big(-C-C\kappa_{1}^2-\sigma_{k}^{pq,rs}h_{pql}h_{rsl}+\sum_sh_{sll}d_{\nu}\psi(\tau_s)\Big)\nonumber\\
&&+\sum_l\sigma_{k}^{ii}e^{\kappa_l}h_{lli}^2+\s_k^{ii}\sum_{p\neq
q}\frac{e^{\kappa_{p}}
-e^{\kappa_{q}}}{\kappa_{p}-\kappa_{q}}h_{pq i}^2-\Big(\frac{1}{P}+\frac{1}{P\log P}\Big)\sigma_{k}^{ii}P_i^2\bigg]\nonumber\\
&&-N\sigma_{k}^{ii}\kappa_{i}^2\lt<\nu, E\rt>+\sum_sNu_s\sigma_{k}^{ii}h_{sii}-\sigma_k^{ii}\kappa_i^2.\nonumber
\end{eqnarray}
By (\ref{dp2.11}) and (\ref{dp2.14}), we have
\begin{eqnarray}\label{dp2.17}
\frac{1}{P\log P}\sum_s\sum_le^{\kappa_l}h_{sll}d_{\nu}\psi(\tau_s)+\sum_sNu_s\sigma_{k}^{ii}h_{sii}\geq-C.\nonumber
\end{eqnarray}

Now, for any constant $K>1$, we denote
\begin{eqnarray}
&&A_i=e^{\kappa_i}\Big[K(\sigma_{k})_i^2-\sum_{p\neq
q}\sigma_{k}^{pp,qq}h_{ppi}h_{qqi}\Big], \ \  B_i=2\sum_{l\neq
i}\sigma_{k}^{ii,ll}e^{\kappa_l}h_{lli}^2, \nonumber \\
&&C_i=\sigma_{k}^{ii}\sum_le^{\kappa_l}h_{lli}^2, \  \
D_i=2\sum_{l\neq
i}\sigma_{k}^{ll}\frac{e^{\kappa_l}-e^{\kappa_i}}{\kappa_l-\kappa_i}h_{lli}^2,
\ \ E_i=\frac{1+\log P}{P\log P}\sigma_{k}^{ii}P_i^2\nonumber.
\end{eqnarray}
Combinning
$$
-\sum_l\sigma_{k}^{pq,rs}h_{pql}h_{rsl}=\sum_{p\neq
q}\sigma_{k}^{pp,qq}h_{pql}^2-\sum_{p\neq
q}\sigma_{k}^{pp,qq}h_{ppl}h_{qql},
$$
with (\ref{dp2.16}), we get
\begin{eqnarray}
&&\sigma_{k}^{ii}\phi_{ii}\label{dp2.18}\\
&\geq &\frac{1}{P\log P}\dsum_i(A_i+B_i+C_i+D_i-E_i)\nonumber\\
&&+(-N\lt<\nu, E\rt>-1)\sigma_{k}^{ii}\kappa_{i}^2-C\kappa_{1}.\nonumber
\end{eqnarray}
\par

\begin{claim}\label{claim 1}
For any given $ 0<\varepsilon<\frac{1}{2}$, we let
$\alpha=\frac{1-2\varepsilon}{1+\varepsilon}$. There exists a
positive constant $\delta<\dfrac{1}{2}$ such that, for any
$|\kappa_i|\leq \delta \kappa_1, 1\leq i\leq n$, if the constant $K$ and the maximum principal curvature $\kappa_1$ both are
sufficiently large, we have
\begin{align*}
A_i+B_i+C_i+D_i-E_i-\frac{\alpha}{P\log
P}\sigma_{k}^{ii}P_i^2\geq0.
\end{align*}
\end{claim}
Applying Lemma 6 in \cite{RW}, we can see that when $K$ is chosen to be sufficiently large, then $A_i\geq 0.$
By the Cauchy-Schwarz inequality, we have
\begin{eqnarray}\label{dp2.19}
P_i^2&= & e^{2\kappa_i}h_{iii}^2 +2 \sum_{l\neq
i}e^{\kappa_i+\kappa_l}h_{iii}h_{lli}
+ \Big(\sum_{l\neq i}e^{\kappa_l}h_{lli}\Big)^2\\
&\leq & e^{2\kappa_i}h_{iii}^2 +2 \sum_{l\neq
i}e^{\kappa_i+\kappa_l}h_{iii}h_{lli}+
(P-e^{\kappa_i})\sum_{l\neq i}e^{\kappa_l}h_{lli}^2. \nonumber
\end{eqnarray}
Thus,
\begin{eqnarray}\label{dp2.20}
&&B_i+C_i+D_i-E_i-\frac{\alpha}{P\log
P}\sigma_{k}^{ii}P_i^2\\
&\geq&2\sum_{l\neq i}e^{\kappa_l}\sigma_{k}^{ll,ii}h_{lli}^2 +
2\sum_{l\neq
i}\frac{e^{\kappa_l}-e^{\kappa_i}}{\kappa_l-\kappa_i}\sigma_{k}^{ll}h_{lli}^2
-\frac{1+\alpha}{\log P}\sum_{l\neq
i}e^{\kappa_l}\sigma_{k}^{ii}h_{lli}^2\nonumber\\
&&+\frac{1+\alpha+\log P}{P\log P}\sum_{l\neq
i}e^{\kappa_l+\kappa_i}\sigma_{k}^{ii}h_{lli}^2+e^{\kappa_i}\sigma_{k}^{ii}h_{iii}^2 \nonumber\\
&&-\frac{1+\alpha+\log P}{P\log
P}e^{2\kappa_i}\sigma_{k}^{ii}h_{iii}^2 -2\frac{1+\alpha+\log
P}{P\log P}\sum_{l\neq
i}e^{\kappa_i+\kappa_l}\sigma_{k}^{ii}h_{iii}h_{lli}  \nonumber.
\end{eqnarray}
Let $\varepsilon$ be equal to the $\varepsilon_T$ in Lemma 12 of \cite{RW}. Then we know there
exists a positive constant $\delta<\varepsilon$ such that, when $|\ka_i|<\delta\ka_1$
\begin{align}\label{dp2.21}
(2-\varepsilon)\sum_{l\neq
i}e^{\kappa_l}\sigma_{k}^{ll,ii}h_{lli}^2 +
(2-\varepsilon)\sum_{l\neq
i}\frac{e^{\kappa_l}-e^{\kappa_i}}{\kappa_l-\kappa_i}\sigma_{k}^{ll}h_{lli}^2
-\dfrac{1+\alpha}{\log P}\sum_{l\neq
i}e^{\kappa_l}\sigma_{k}^{ii}h_{lli}^2\geq 0.
\end{align}
On the other hand, we have
\begin{align}\label{CauS}
&\sum_{l\neq i,1}e^{\kappa_l+\kappa_i}\sigma_{k}^{ii}h_{lli}^2
-2\sum_{l\neq
i,1}e^{\kappa_i+\kappa_l}\sigma_{k}^{ii}h_{iii}h_{lli}\geq-\sum_{l\neq
i,1}e^{\kappa_l+\kappa_i}\sigma_{k}^{ii}h_{iii}^2.
\end{align}
It follows
\begin{eqnarray}\label{dp2.23}
&&B_i+C_i+D_i-E_i-\frac{\alpha}{P\log
P}\sigma_{k}^{ii}P_i^2\\
&\geq&\frac{1+\alpha+\log P}{P\log
P}e^{\kappa_1+\kappa_i}\sigma_{k}^{ii}h_{11i}^2+e^{\kappa_i}\sigma_{k}^{ii}h_{iii}^2 \nonumber\\
&&-\frac{1+\alpha+\log P}{P\log P}\sum_{l\neq
1}e^{\kappa_l+\kappa_i}\sigma_{k}^{ii}h_{iii}^2
-2\frac{1+\alpha+\log P}{P\log P}e^{\kappa_i+\kappa_1}\sigma_{k}^{ii}h_{iii}h_{11i} \nonumber\\
&&+\varepsilon e^{\kappa_1}\sigma_{k}^{11,ii}h_{11i}^2 +
\varepsilon\frac{e^{\kappa_1}-e^{\kappa_i}}{\kappa_1-\kappa_i}\sigma_{k}^{11}h_{11i}^2.
\nonumber
\end{eqnarray}

A straightforward calculation shows that when $\ka_1$ is very large the following inequalities hold:
\begin{eqnarray}
e^{\kappa_i}\sigma_{k}^{ii}h_{iii}^2 -\frac{1+\alpha+\log
P}{P\log P}\sum_{l\neq
1}e^{\kappa_l+\kappa_i}\sigma_{k}^{ii}h_{iii}^2&\geq&
\Big(\frac{e^{\kappa_1}}{P}-\frac{1+\alpha}{\log
P}\Big)e^{\kappa_i}\sigma_{k}^{ii}h_{iii}^2 \nonumber\\
&\geq& \frac{1}{n+1}e^{\kappa_i}\sigma_{k}^{ii}h_{iii}^2,\nonumber
\end{eqnarray}
and
\begin{eqnarray}
-2\frac{1+\alpha+\log P}{P\log
P}e^{\kappa_i+\kappa_1}\sigma_{k}^{ii}|h_{iii}h_{11i}| &\geq &
-\frac{3}{P}e^{\kappa_i+\kappa_1}\sigma_{k}^{ii}|h_{iii}h_{11i}|\nonumber\\
&\geq& -3e^{\kappa_i}\sigma_{k}^{ii}|h_{iii}h_{11i}|.\nonumber
\end{eqnarray}
Moreover, it is easy to see that
\begin{align}\label{dp2.24}
e^{\kappa_1}\sigma_{k}^{11,ii}h_{11i}^2 +
\frac{e^{\kappa_1}-e^{\kappa_i}}{\kappa_1-\kappa_i}\sigma_{k}^{11}h_{11i}^2
=e^{\kappa_i}\sigma_{k}^{11,ii}h_{11i}^2 +
\frac{e^{\kappa_1}-e^{\kappa_i}}{\kappa_1-\kappa_i}\sigma_{k}^{ii}h_{11i}^2.
\end{align}
By the Taylor expansion, we have
\begin{align}\label{e2.25}
\frac{e^{\kappa_1}-e^{\kappa_i}}{\kappa_1-\kappa_i}\sigma_{k}^{ii}h_{11i}^2
=e^{\kappa_i}\sum_{m\geq
1}\frac{(\kappa_1-\kappa_i)^{m-1}}{m!}\sigma_{k}^{ii}h_{11i}^2.
\end{align}
Combining the previous four formulae with (\ref{dp2.23}), we obtain
when $\kappa_1$ is sufficiently large and $|\ka_i|<\delta\ka_1,$
\begin{eqnarray}
&&B_i+C_i+D_i-E_i-\frac{\alpha}{P\log
P}\sigma_{k}^{ii}P_i^2\nonumber\\
&\geq&e^{\kappa_i}\sigma_{k}^{ii}\Big[\frac{1}{n+1}h_{iii}^2-3|h_{iii}h_{11i}|
+\varepsilon\sum_{m\geq
1}\frac{(\kappa_1-\kappa_i)^{m-1}}{m!}h_{11i}^2\Big] \nonumber\\
&\geq& 0.\nonumber
\end{eqnarray}
Therefore, Claim 1 is proved.

Now, recall Section 4 of \cite{RW} and the proof of Theorem 14 in \cite{RW1}, we know the following claim is true.
\begin{claim}\label{claim 2}
Suppose $k=n-1$ ($n\geq 3$) and $k=n-2$ ($n\geq 5$). For any index  $1\leq i\leq n$,
if the positive constant $K$ and the maximum principal curvature $\kappa_1$
both are sufficiently large,
we have
\begin{align*}
A_i+B_i+C_i+D_i-E_i\geq 0.
\end{align*}
\end{claim}
By Claim \ref{claim 1} and Claim \ref{claim 2}, \eqref{dp2.18} becomes
\be\label{newdp51}
0\geq \sum_{|\kappa_i|<\delta \kappa_1}\frac{\alpha}{(P\log P)^2}\sigma_{k}^{ii}P_i^2+(-N\lt<\nu, E\rt>-1)\sigma_{k}^{ii}\kappa_{i}^2-C\kappa_{1}.
\ee
Here, the constant $\delta$ is the constant chosen in Claim 1.
Choose $N>0$ such that $\sigma_{k}^{11}\kappa_{1}^2(-N\lt<\nu, E\rt>-1)-C\kappa_{1}> 0,$ we get a contradiction.
Therefore, our desired estimate follows immediately.
\end{proof}
By Lemma \ref{lem gradient}, Lemma \ref{lem c2 boundary}, and Lemma \ref{lem c2 global}, we conclude that when $k=n-1, n-2,$ the Dirichlet problem \eqref{Dirichlet curvature}
admits a $k$-convex solution.

\bigskip
\section{The local estimates}
\label{local C1}
We will devote this section to establishing the local $C^1$ and $C^2$ estimates for the solution $u$ of \eqref{curvature1}.
\subsection{Local $C^1$ estimates}
In this subsection, we will prove the local $C^1$ estimate. We will split it into two cases. In the first case, we will assume $u$ is a
convex solution of \eqref{curvature}; in the second case, we will assume  $u$ is a $k$-convex solution of \eqref{curvature-k}.
Note that in both cases  our results hold for $1\leq k\leq n.$

For strictly convex, spacelike hypersurfaces, Bayard-Schn\"urer \cite{BS} proved the following local gradient estimate lemma.
\begin{lemm}
\label{lc1lem1}(Lemma 5.1 in \cite{BS})
Let $\Omega\subset \R^n$ be a bounded open set. Let $u, \bar{u}, \Psi:\Omega\rightarrow\mathbb{R}^n$ be strictly spacelike. Assume that $u$ is strictly convex and $u<\uu$ in $\Omega.$ Also assume that near $\partial\Omega,$ we have $\Psi>\bar{u}.$
Consider the set, where $u>\Psi.$ For every $x$ in this set, we have the following gradient estimate for $u:$
\[\frac{1}{\sqrt{1-|Du|^2}}\leq\frac{1}{u(x)-\Psi(x)}\cdot\sup\limits_{\{u>\Psi\}}\frac{\bar{u}-\Psi}{\sqrt{1-|D\Psi|^2}}.\]
\end{lemm}
For $k$-convex, spacelike hypersurfaces, Bayard \cite{Bayard05} proved a similar result when $k=2$. In the following, we will extend it to all $k.$ Our argument is a modification of Bayards' in \cite{Bayard05}. We would also like to mention that the basic idea of this argument had appeared in Chow-Wang \cite{CW}.
\begin{lemm}
\label{lem-local c1}
Let $\Omega\subset \R^n$ be a bounded open set. Let $u, \bar{u}, \Psi:\Omega\rightarrow\mathbb{R}^n$ be strictly spacelike. Assume that $\M_u=\{(x,u(x))|x\in\Omega\}$ is a k-convex hypersurface satisfying
$$\sigma_k(\kappa[\M_u])=\psi(x,u(x))$$
and $u\leq\bar{u}$ in $\Omega$.
Also assume that near $\partial\Omega,$ we have $\Psi>\bar{u}.$
Consider the set, where $u>\Psi.$ For every $x$ in this set, we have the following gradient estimate for $u:$
\[\frac{1}{\sqrt{1-|Du|^2}}\leq\lt[\frac{1}{u(x)-\Psi(x)}\cdot\sup\limits_{\{u>\Psi\}}(\bar{u}-\Psi)\rt]^NC.\]
Here, $N=N(n, k)$ is a uniform constant only depending on $n,k,$ and $C=C(\uu-\Psi, |\Psi|_{C^2}, |\psi|_{C^1})$
is a uniform constant depending on the upper bound of $\uu-\Psi$, $\frac{1}{\sqrt{1-|D\Psi|^2}}$, $D^2\Psi,$ and $|\psi|_{C^1}$.
\end{lemm}
\begin{proof}
Consider the test function:
$$\phi=(u-\Psi)^N(-\lt<\nu, E\rt>),$$ where $N$ is a large undetermined constant. Assume the function $\phi$ achieves its maximum at $P.$
We may choose a local normal coordinate $\{\tau_1, \cdots, \tau_n\}$ such that at $P,$ $h_{ij}=\ka_i\delta_{ij}.$ Differentiating
$\phi$ twice at $P,$ we have,
\begin{eqnarray}\label{40}
0&=&\frac{\phi_i}{\phi}=N\frac{u_i-\Psi_i}{u-\Psi}+\frac{h_{im}u_m}{-\lt<\nu, E\rt>},\\
0&\geq&\frac{\phi_{ii}}{\phi}-\frac{\phi_i^2}{\phi^2}=N\frac{u_{ii}-\Psi_{ii}}{u-\Psi}-N\frac{(u_i-\Psi_i)^2}{(u-\Psi)^2}\nonumber\\
&&+\frac{\sum_mh_{im}^2(-\lt<\nu, E\rt>)+\sum_mh_{imi}u_m}{-\lt<\nu, E\rt>}-\frac{(\sum_mh_{im}u_m)^2}{(-\lt<\nu, E\rt>)^2}\nonumber
\end{eqnarray}
Contracting with $\sigma_k^{ii}$, we get
\begin{eqnarray}\label{400}
0\geq \frac{\sigma^{ii}_k\phi_{ii}}{\phi}&=&N\frac{\sigma_k^{ii}u_{ii}-\sigma_k^{ii}\Psi_{ii}}{u-\Psi}-N\frac{\sigma_k^{ii}(u_i-\Psi_i)^2}{(u-\Psi)^2}\\
&&+\sigma^{ii}_k\kappa_{i}^2+\frac{\sigma_k^{ii}\sum_mh_{iim}u_m}{-\lt<\nu, E\rt>}-\frac{\sigma_k^{ii}\kappa_i^2u_i^2}{(-\lt<\nu, E\rt>)^2}\nonumber
\end{eqnarray}

Without loss of generality, we may assume that at $P$
$$u_1^2\geq \frac{|\nabla u|^2}{n},$$
where $\nabla$ is the Levi-Civita connection on $\M_u.$
By \eqref{40}, we have
\begin{eqnarray}
\kappa_1=\frac{N\lt<\nu, E\rt>}{u-\Psi}\left(1-\frac{\Psi_1}{u_1}\right).\nonumber
\end{eqnarray}
We may also assume $|\nabla u(P)|$ is so large that $|\frac{\Psi_1}{u_1}|<\frac{1}{2}$. Then at $P$ we can see,
\begin{eqnarray}\label{41}
\kappa_1<\frac{N}{2}\frac{\lt<\nu, E\rt>}{u-\Psi}.
\end{eqnarray}
Thus, if $N$ is sufficiently large, $\kappa_1$ is negative and its norm is large. Using the inequality (26) in Lin-Trudinger \cite{LT}, we obtain
$$\sum_{i\geq 2}\sigma_k^{ii}\kappa_{i}^2\geq \eta \sigma^{11}_k\kappa_1^2,$$
where $\eta$ is a uniform constant only depending on $n,k$. Therefore,
\begin{eqnarray}
\sigma_k^{ii}\kappa_{i}^2-\frac{\sigma_k^{ii}\kappa_i^2u^2_i}{(-\lt<\nu, E\rt>)^2}\geq \nonumber
\sum\limits_{i\geq 2}\s_k^{ii}\ka_i^2-\lt(1-\frac{1}{n}\rt)\sum\limits_{i\geq 2}\s_k^{ii}\ka_i^2\geq\frac{\eta}{n}\s_k^{11}\ka_1^2\nonumber
:=\eta_0\s_k^{11}\ka_1^2.\nonumber
\end{eqnarray}
By \eqref{41}, we get
\begin{eqnarray}\label{42}
\sigma_k^{ii}\kappa_{i}^2-\frac{\sigma_k^{ii}\kappa_i^2u^2_i}{(-\lt<\nu, E\rt>)^2}\geq \frac{\eta_0 N^2}{4}
\sigma^{11}_k\frac{(-\lt<\nu, E\rt>)^2}{(u-\Psi)^2}.
\end{eqnarray}
Inserting \eqref{curvature} and \eqref{42} into \eqref{400} yields,
\begin{eqnarray}\label{43}
0&\geq& N(u-\Psi)[\sigma_k^{ii}\kappa_{i}(-\lt<\nu, E\rt>)-\sigma_k^{ii}\Psi_{ii}]-N\sigma_k^{ii}(u_i-\Psi_i)^2\\
&&+(u-\Psi)^2\frac{\sum_m\psi_mu_m}{-\lt<\nu, E\rt>}+\frac{\eta_0 N^2}{4} \sigma^{11}_k(-\lt<\nu, E\rt>)^2.\nonumber
\end{eqnarray}
Notice that
\[\psi_m=\sum\limits_{l=1}^n\psi_{x_l}\lt<\tau_m, \frac{\p}{\p x_l}\rt>+\psi_u\lt<-\tau_m, E\rt>,\]
we calculate,
\be\label{44}
\frac{\sum_m\psi_mu_m}{-\lt<\nu, E\rt>}\geq-C\lt(1+\lt<-\nu, E\rt>\rt).
\ee
Combing \eqref{43} with \eqref{44}, we get
\begin{eqnarray}\label{45}
\\
0&\geq&-(n-k+1)N(\uu-\Psi)\sigma_{k-1}|\nabla^2 \Psi|-2(n-k+1)N\sigma_{k-1}(|\nabla u|^2+|\nabla \Psi|^2)\nonumber\\
&&-C(\uu-\Psi)^2\lt(1+\lt<-\nu, E\rt>\rt)+\frac{\eta_0 N^2}{4} \sigma^{11}_k(-\lt<\nu, E\rt>)^2.\nonumber
\end{eqnarray}
Notice that when $\kappa_1<0$, we have
$$\sigma_{k-1}=\kappa_1\sigma_{k-2}(\kappa|1)+\sigma_{k-1}(\kappa|1)\leq \sigma_k^{11}.$$
Moreover, $-\lt<\nu, E\rt>=\sqrt{1+|\nabla u|^2}.$ Let $N$ be sufficiently large in \eqref{45}, we obtain the desired estimate.
\end{proof}

\subsection{The Pogorelov type local $C^2$ estimates }
Recall that in \cite{WX} (see Lemma 24), we proved the Pogorelove type local $C^2$ estimate for strictly convex, spacelike, constant $\s_k$ curvature hypersurfaces.
With small modifications, we can show
\begin{lemm}
\label{lem-local c2}
Let $\urs$ be the solution of \eqref{D-approximate} and $\ur$ be the Legendre transform of $\urs.$
For any given $s>2C_0+1,$ where $C_0>\min\uu$ is an arbitrary constant,  let $r_s>0$ be a positive number such that when $r>r_s,$
$\ur|_{\partial\Omega_r}>s,$ where $\Omega_r=D\urs(B_r).$ Let $\ka_{\max}(x)$ be the largest principal curvature of $\M_{\ur}$ at $x,$
where $\M_{\ur}=\{(x, \ur(x)) |\,x\in\Omega_r\}.$ Then, for $r>r_s$ we have
\be\label{local c2}\max_{\M_{\ur}}(s-\ur)\la_{\max}\leq C.\ee
Here, $C$ depends on the local $C^1$ estimates of $\ur$ and $s.$
\end{lemm}

In the rest of this subsection, we will establish the Pogorelov type local $C^2$ estimates for the $k$-convex solution of
equation \eqref{curvature}, where $k=n-1\,(n\geq 3), n-2\,(n\geq 5).$

\begin{lemm}
\label{lem-local c2*}
Let $u^n$ be the $k$-convex solution of \eqref{Dirichlet curvature} with $\psi=\psi(X, \nu)$, where $k=n-1\,(n\geq 3), n-2\,(n\geq 5).$
For any given $s>1,$ let $m>s,$ then
$u^m|_{\partial \Omega_m}=m>s.$ Let $\ka_{\max}(x)$ be the largest principal curvature of $\M_{u^m}$ at $x,$
where $\M_{u^m}=\{(x, u^m(x)) |\,x\in \Omega_m\}.$ Then, for $m>s$ we have
\[\max_{\M_{u^m}}(s-u^m)\ka_{\max}\leq C.\]
Here, $C$ depends on the local $C^1$ estimates of $u^m$ and $s.$
\end{lemm}
\begin{proof}
In this proof, for our convenience when there is no confusion, we will drop the superscript on $u^m$.
Now, on $\Omega_m$, we consider the following test function whose form first appeared in
\cite{GRW},
$$
\phi=\beta\log (s-u)+\log\log P-N\lt<\nu, E\rt>.
$$
Here the function $P$ is defined by $$P=\dsum_le^{\kappa_l}, $$ and $\beta, N$ are constants to be determined later.

Let $U_s=\{x\in \R^n| u(x)<s\},$ we may assume that the maximum of $\phi$ is achieved at $P_0\in U_s$.
Choose a local normal coordinate $\{\tau_1,\tau_2,\cdots,\tau_n\}$ such that at $P_0,$ $h_{ij}=\ka_i\delta_{ij}$ and
$\kappa_1\geq \kappa_{2}\cdots\geq \kappa_{n}.$

Differentiating the function $\phi$ twice  at  $P_0$, we get
\begin{equation}\label{e2.11}
\phi_i = -\frac{\beta u_i}{s-u}+\dfrac{P_i}{P\log P}+Nh_{ii}u_i = 0,
\end{equation}
and
\begin{eqnarray*}
0&\geq&\phi_{ii}\\
&=& \frac{P_{ii}}{P\log P}-\frac{P_i^2}{P^2\log P}-\frac{P_i^2}{(P\log P)^2}+ \frac{\beta h_{ii}\lt<\nu, E\rt>}{s-u}-\frac{\beta u_i^2}{(s-u)^2}\nonumber\\
&&-Nh_{ii}^2\lt<\nu, E\rt>+\sum_sNu_sh_{isi}\nonumber\\
&=&\frac{1}{P\log
P}\bigg[\sum_le^{\kappa_l}h_{llii}+\sum_le^{\kappa_l}h_{lli}^2+\sum_{p\neq
q}\frac{e^{\kappa_{p}}
-e^{\kappa_{q}}}{\kappa_{p}-\kappa_{q}}h_{pq i}^2-\Big(\frac{1}{P}+\frac{1}{P\log P}\Big)P_i^2\bigg]\nonumber\\
&&+ \frac{\beta h_{ii}\lt<\nu, E\rt>}{s-u}-\frac{\beta u_i^2}{(s-u)^2}-Nh_{ii}^2\lt<\nu, E\rt>+\sum_sNu_sh_{iis}\nonumber
\end{eqnarray*}
Contracting with $\sigma_{k}^{ii}$, we have
\begin{eqnarray}\label{e2.13}
&&\sigma_{k}^{ii}\phi_{ii}\\
&=&\frac{\sigma_{k}^{ii}}{P\log
P}\bigg[\sum_le^{\kappa_l}h_{llii}+\sum_le^{\kappa_l}h_{lli}^2+\sum_{p\neq
q}\frac{e^{\kappa_{p}}
-e^{\kappa_{q}}}{\kappa_{p}-\kappa_{q}}h_{pq i}^2-\Big(\frac{1}{P}+\frac{1}{P\log P}\Big)P_i^2\bigg]\nonumber\\
&&+\frac{\beta \sigma_{k}^{ii}\kappa_{i}\lt<\nu, E\rt>}{s-u}-\frac{\beta
\sigma_{k}^{ii}u_i^2}{(s-u)^2}-N\sigma_{k}^{ii}\kappa_{i}^2\lt<\nu, E\rt>+\sum_sNu_s\sigma_{k}^{ii}h_{iis}.\nonumber
\end{eqnarray}\par
At $P_0$, differentiating the equation (\ref{curvature}) twice yields,
\begin{align}\label{e2.14}
\sigma_{k}^{ii}h_{iil}=d_X\psi(\tau_l)+\kappa_ld_{\nu}\psi(\tau_l),
\end{align}
and
\begin{align}\label{e2.15}
\sigma_{k}^{ii}h_{iill}+\sigma_{k}^{pq,rs}h_{pql}h_{rsl}\geq
-C-Ch_{11}^2+\sum_sh_{sll}d_{\nu}\psi(\tau_s),
\end{align}
where $C$ is some uniform constant. Note that
\begin{eqnarray}\label{new216}
h_{llii}=h_{iill}-h_{ii}h_{ll}^2+h^2_{ii}h_{ll}.
\end{eqnarray}
Inserting  (\ref{e2.15}) and \eqref{new216} into (\ref{e2.13}), we obtain
\begin{eqnarray}
&&\sigma_{k}^{ii}\phi_{ii}\label{e2.16}\\
&\geq &\frac{1}{P\log P}\bigg[\sum_le^{\kappa_l}\Big(-C-C\kappa_{1}^2-\sigma_{k}^{pq,rs}h_{pql}h_{rsl}+\sum_sh_{sll}d_{\nu}\psi(\p_s)\Big)\nonumber\\
&&+\sum_l\sigma_{k}^{ii}e^{\kappa_l}h_{lli}^2+\s_k^{ii}\sum_{p\neq
q}\frac{e^{\kappa_{p}}
-e^{\kappa_{q}}}{\kappa_{p}-\kappa_{q}}h_{pq i}^2-\Big(\frac{1}{P}+\frac{1}{P\log P}\Big)\sigma_{k}^{ii}P_i^2\bigg]\nonumber\\
&&+\frac{\beta k\sigma_{k}\lt<\nu, E\rt>}{s-u}-\frac{\beta
\sigma_{k}^{ii}u_i^2}{(s-u)^2}-N\sigma_{k}^{ii}\kappa_{i}^2\lt<\nu, E\rt>+\sum_sNu_s\sigma_{k}^{ii}h_{sii}-\sigma_k^{ii}\kappa_i^2.\nonumber
\end{eqnarray}
From (\ref{e2.11}) and (\ref{e2.14}), we deduce
\begin{eqnarray}\label{e2.17}
\frac{1}{P\log P}\sum_j\sum_l
e^{\kappa_l}h_{jll}d_{\nu}\psi(\tau_j)+\sum_jNu_j\sigma_{k}^{ii}h_{sii}\geq\sum_ld_{\nu}\psi(\tau_l)\frac{\beta
u_l}{s-u}-C.\nonumber
\end{eqnarray}

For any constant $K>1$, denote
\begin{eqnarray}
&&A_i=e^{\kappa_i}\Big[K(\sigma_{k})_i^2-\sum_{p\neq
q}\sigma_{k}^{pp,qq}h_{ppi}h_{qqi}\Big], \ \  B_i=2\sum_{l\neq
i}\sigma_{k}^{ii,ll}e^{\kappa_l}h_{lli}^2, \nonumber \\
&&C_i=\sigma_{k}^{ii}\sum_le^{\kappa_l}h_{lli}^2, \  \
D_i=2\sum_{l\neq
i}\sigma_{k}^{ll}\frac{e^{\kappa_l}-e^{\kappa_i}}{\kappa_l-\kappa_i}h_{lli}^2,
\ \ E_i=\frac{1+\log P}{P\log P}\sigma_{k}^{ii}P_i^2\nonumber.
\end{eqnarray}
Note that
$$
-\sum_l\sigma_{k}^{pq,rs}h_{pql}h_{rsl}=\sum_{p\neq
q}\sigma_{k}^{pp,qq}h_{pql}^2-\sum_{p\neq
q}\sigma_{k}^{pp,qq}h_{ppl}h_{qql},
$$
Therefore, \eqref{e2.16} becomes
\begin{eqnarray}
&&\sigma_{k}^{ii}\phi_{ii}\label{e2.18}\\
&\geq &\frac{1}{P\log P}\dsum_i(A_i+B_i+C_i+D_i-E_i)\nonumber\\
&&+\frac{\beta k\sigma_{k}\lt<\nu, E\rt>}{s-u}-\frac{\beta
\sigma_{k}^{ii}u_i^2}{(s-u)^2}+(-N\lt<\nu, E\rt>-1)\sigma_{k}^{ii}\kappa_{i}^2+\sum_ld_{\nu}\psi(\tau_l)\frac{\beta
u_l}{s-u}-C\kappa_{1}.\nonumber
\end{eqnarray}

Following the same argument as the one in the proof of Lemma \ref{lem c2 global},
from \eqref{e2.18} we obtain,
\be\label{new51}
\begin{aligned}
0&\geq \sum_{|\kappa_i|<\delta \kappa_1}\frac{\alpha}{(P\log P)^2}\sigma_{k}^{ii}P_i^2\\
&+\frac{\beta k\sigma_{k}\lt<\nu, E\rt>}{s-u}-\frac{\beta\sigma_{k}^{ii}u_i^2}{(s-u)^2}+(-N\lt<\nu, E\rt>-1)\sigma_{k}^{ii}
\kappa_{i}^2+\sum_ld_{\nu}\psi(\tau_l)\frac{\beta
u_l}{s-u}-C\kappa_{1}.
\end{aligned}
\ee
Here, the constant $\delta$ is the same constant as the one chosen in Claim 1 of Lemma \ref{lem c2 global}.
Moreover, by (\ref{e2.11}), we have
\begin{align*}
-\frac{\beta \sigma_{k}^{ii}u_i^2}{(s-u)^2}\geq
-\frac{\sigma_{k}^{ii}}{\beta}\Big[2\Big(\frac{P_i}{P\log
P}\Big)^2+2N^2u_i^2\kappa_{i}^2\Big].
\end{align*}
Choose $\beta>0$ such that $\alpha\beta>2$, then \eqref{new51} implies
\begin{align}
0\geq &\frac{\beta k\sigma_{k}\lt<\nu, E\rt>}{s-u}-\sum_{|\kappa_i|\geq
\delta \kappa_1}\frac{\beta
\sigma_{k}^{ii}u_i^2}{(s-u)^2}\label{new52}\\
&+(-N\lt<\nu, E\rt>-1)\sigma_{k}^{ii}\kappa_{i}^2+\sum_ld_{\nu}\psi(\tau_l)\frac{\beta
u_l}{s-u}-C\kappa_{1}-\sum_{|\kappa_i|< \delta
\kappa_1}\frac{\sigma_{k}^{ii}}{\beta}2N^2u_i^2\kappa_{i}^2.\nonumber
\end{align}
Now, first choose $N>0$ such that $\frac{1}{2}\sum_{|\kappa_i|\geq \delta
\kappa_1}\sigma_{k}^{ii}\kappa_{i}^2(-N\langle\nu, E\rangle-1)-C\kappa_{1}\geq 0$, then choose
$\beta=\beta(N)$ sufficiently large such that $\sum_{|\kappa_i|< \delta
\kappa_1}\Big(\sigma_{k}^{ii}\kappa_{i}^2(-N\langle\nu, E\rangle-1)-\frac{\sigma_{k}^{ii}}{\beta}2N^2u_i^2\kappa_{i}^2\Big)\geq
0.$ We deduce
\begin{align}\label{e2.27}
\frac{\beta C}{s-u}+\sum_{|\kappa_i|\geq \delta \kappa_1}\frac{2\beta
\sigma_{k}^{ii}u_i^2}{(s-u)^2}\geq &\sum_{|\kappa_i|\geq \delta
\kappa_1}\sigma_{k}^{ii}\kappa_{i}^2(-N\langle\nu, E\rangle-1).
\end{align}
If $\frac{C}{s-u}\geq \sum_{|\kappa_i|\geq \delta
\kappa_1}\frac{2\beta \sigma_{k}^{ii}u_i^2}{(s-u)^2}$, we get
\begin{align*}
\frac{2C\beta}{s-u}\geq &\sigma_{k}^{11}\kappa_{1}^2(-N\lt<\nu, E\rt>-1)\geq c_0(N-1)\kappa_1,
\end{align*}
which implies the desired estimate.
If $\frac{C}{s-u}\leq \sum_{|\kappa_i|\geq \delta
\kappa_1}\frac{2\beta \sigma_{k}^{ii}u_i^2}{(s-u)^2}$, we let $i_0$ denote the index
of the maximum value element of the set
$$\Big\{\frac{2\beta
\sigma_{k}^{ii}u_i^2}{(s-u)^2}; |\kappa_i|\geq \delta
\kappa_1\Big\}.$$
Then, we obtain
\begin{align*}
4n\frac{\beta \sigma_{k}^{i_0i_0}u_{i_0}^2}{(s-u)^2}\geq
&\sigma_{k}^{i_0i_0}\kappa_{i_0}^2(-N\lt<\nu, E\rt>-1)\geq
C(N-1)\sigma_{k}^{i_0i_0}\delta^2\kappa_{1}^2,
\end{align*}
which also implies our desired estimate.
\end{proof}

\section{The prescribed curvature problem}
\label{prescribed curvature}
We will prove Theorem \ref{theo1} and \ref{theo2} in this section.

Let's consider the proof of Theorem \ref{theo1} first.
Recall that in Subsection \ref{Dirichlet sub1}, we have solved the approximate Dirichlet problem \eqref{D-approximate} on $B_r,$ for $r<1$. We will denote the strictly convex solution of \eqref{D-approximate} by $\urs$. We further denote the Legendre transform of $(B_r,\urs)$ to be $(\Omega_r, \ur)$, where
$\Omega_r=D\urs(B_r)$ is the domain of $\ur$. By Lemma 19 and 20 in \cite{WX} we have
\begin{eqnarray}\label{new61}
\lu\leq \ur\leq \uu,
\end{eqnarray}
in $\Omega_r$.

In the following, we will denote $\tilde{\Omega}_r=D\lu^*(B_r)$ to be the domain of $\lu_r:=\lu|_{\tilde{\Omega}_r}$. It is not difficult to see that these domains are increasing, namely,
$$\tilde{\Omega}_r\subset\tilde{\Omega}_s, \text{ for } r<s.$$ Moreover, by the choice of $\lu$ in Subsection \ref{Dirichlet sub1}, we have
$$\lu|_{\p\tilde{\Omega}_r}\rightarrow +\infty, \text{ as } r\rightarrow 1.$$
Thus, by the comparison principle, we have
\be\label{cv1.18}
\begin{aligned}
u_r|_{\T\Omega_r}&=[\xi\cdot D\us_r(\xi)-\us_r(\xi)]|_{\T B_r}\\
&\geq [\xi\cdot D\lu^*(\xi)-\lu^*(\xi)]|_{\T B_r}\\
&= \lu|_{\p \tilde{\Omega}_r}.
\end{aligned}
\ee
From this we can see that, as $r\rightarrow 1$, $u_r|_{\T\Omega_r}\rightarrow +\infty.$ This in turn implies,
for any compact set $\K\subset\R^n,$ there exists a constant $c_\K=c(\K)<1$ such that when $r>c_\K$, $\Omega_r\supset \K.$
Therefore,  for any compact set $\K\subset\R^n,$ we can apply Lemma \ref{lc1lem1} and Lemma \ref{lem-local c2} to obtain uniform $C^1$ and $C^2$ bounds for $\ur$ in $\K$.

More precisely, in order to obtain the local $C^1$ estimate, we introduce a new subsolution $\lu_1$ of \eqref{curvature}, where $\lu_1$ satisfies
$$\sigma_k(\kappa_1,\cdots,\kappa_n)=c_1+100,$$
and as $|x|\goto\infty$
$$\lu_1\goto |x|+\varphi\left(\frac{x}{|x|}\right).$$
By the strong maximum principle we have, when $x\in\R^n$
$$\lu_1(x)<\lu(x). $$
Thus, for any compact convex domain $\K$, let
$$2\delta=\min_{\K}(\lu-\lu_1).$$ We define a strict spacelike function $\Psi=\lu_1+\delta$. Denote $\K'=\{x\in\mathbb{R}^n; \Psi\leq \uu\}$. Since as $|x|\goto\infty,$ $\lu_1-\uu\goto 0$, we know that $\K'$ is a compact set only depending on $\K$. Applying Lemma \ref{lc1lem1}, for any $(\Omega_r, \ur)$, if $\K'\subset\Omega_r$, we have the gradient estimate:
\[\sup_{\K}\frac{1}{\sqrt{1-|D\ur|^2}}\leq\frac{1}{\delta}\sup\limits_{\K'}\frac{\bar{u}-\Psi}{\sqrt{1-|D\Psi|^2}}.\]

Next, we want to show that for any given compact set $\K\subset\R^n,$ $\{|D^2\ur|\}$ is uniformly bounded in $\K.$
Without loss of generality, let's consider any $B_R\subset\mathbb{R}^n$. Let $C_0=\max_{B_R}\uu$ and $s=2C_0+1$ in Lemma \ref{lem-local c2}. Denote
$U_s=\{x\in\mathbb{R}^n; \lu(x)<s\}$,
then by earlier discussion, it's easy to see that there exists $r_s>0$ such that when $r>r_s,$ $\Omega_r\supset U_s.$ Applying Lemma \ref{lem-local c2}
we obtain when $r>r_s$
$$\sup_{B_R}\kappa_{\max}(M_{\ur})\leq C.$$
Here $C$ depends on the upper bound of $\frac{1}{\sqrt{1-|D\ur|^2}}$ on $\bar{U}_s$, which is independent of $r.$
Using the classical regularity theorem and convergence theorem, we conclude that $(\Omega_r, \ur)$ converges locally smoothly to
an entire, smooth convex function $u$ satisfying \eqref{curvature}. In view of \eqref{new61} and the asymptotic behavior of $\lu,\uu$, we know that
as $|x|\goto\infty,$ $u\goto |x|+\varphi\left(\frac{x}{|x|}\right).$ Moreover, by Remark \ref{rema2} we also know that $u$ is strictly convex. Therefore, its
Gauss map image is $B_1,$ i.e., $Du(\mathbb{R}^n)=B_1$.

Theorem \ref{theo2} follows by replacing Lemma \ref{lc1lem1} and Lemma \ref{lem-local c2} in the proof of Theorem \ref{theo1} with Lemma \ref{lem-local c1} and Lemma \ref{lem-local c2*}.

\section{The radial downward translating soliton}
\label{rs}
In this section, we will study the radially symmetric downward translating soliton. Recall that we say $\M_u$ is a downward translating soliton when its  principal curvatures satisfy
\be\label{soliton*}
\s_k(\ka[\M_u])=\binom{n}{k}\lt(\C-\frac{1}{\sqrt{1-|Du|^2}}\rt)^k,
\ee
where $\C>1$ is a constant. We want to point out that in this section and the next section, $\C$ is the fixed constant in \eqref{soliton*}. We also denote
$$\td{\C}=\sqrt{1-\frac{1}{\C^2}}$$ as in Theorem \ref{theo3}.
The following theorem is a generalization of Theorem 1 in \cite{Ba2020}.
\begin{theo}\label{theo11}
Let $\C>1$ be a positive constant. Then there exits a strictly convex radial solution $u:\mathbb{R}^n\rightarrow \mathbb{R}$ of \eqref{soliton*}, satisfying
$$|Du|\rightarrow \td{\C}, \text{ as } |x|\rightarrow +\infty.$$
Moreover, as $|x|\goto\infty,$ $u(x)$ has the following asymptotic expansion
\be\label{asymptotic}
u(x)=\td{\C}|x|-\frac{1}{\C^2}\sqrt[k]{\frac{n-k}{n}}\log |x|+c_0+o(1)
\ee
for some constant $c_0\in\R.$ In particular, the radial solution $u$ is unique up to the addition of a constant.
\end{theo}
For radial solutions, we will reduce the equation \eqref{soliton*} to an ODE.
Let $u=u(r)$ and $y=\frac{\p u}{\p r}$, then a straightforward calculation yields,
$$D_iu=y\frac{x_i}{|x|}, D^2_{ij}u=\frac{y}{|x|}\left(\delta_{ij}-\frac{x_ix_j}{|x|^2}\right)+y'\frac{x_ix_j}{|x|^2}.$$
Therefore,
$$\ka[\M_u]=\frac{1}{\sqrt{1-y^2}}\left(\frac{y'}{1-y^2},\frac{y}{r},\cdots,\frac{y}{r}\right),$$
and \eqref{soliton*} becomes
\begin{eqnarray}\label{71}
\frac{1}{(1-y^2)^{k/2}}\frac{y^{k-1}}{r^{k-1}}\left(\frac{k}{n}\frac{y'}{1-y^2}+\frac{n-k}{n}\frac{y}{r}\right)=\left(\C-\frac{1}{\sqrt{1-y^2}}\right)^k.
\end{eqnarray}
By a small modification of the proof of Proposition 2.1 in \cite{Ba2020}, we obtain
\begin{prop}\label{prop12}
Under the hypotheses of Theorem \ref{theo11}, there exists a solution $y$ of \eqref{71}, which is defined on $[0,+\infty)$ and smooth on $(0,+\infty),$ such that
$$y(0)=0, 0\leq y <\td{\C} $$
$$\lim_{r\rightarrow +\infty}y(r)=\td{\C}, y'(0)=\C-1, \,\,\text{and}\,\, y'>0 \,\,\text{on}\,\, [0, +\infty).$$
Moreover, as $r\goto0+,$ we have
$$\ka[\M_u(r)]\goto(\C-1)(1,1,\cdots,1).$$
\end{prop}
Since the proof is a small modification of the proof of Proposition 2.1 in \cite{Ba2020}, we skip it here. Now, let's study the asymptotic behavior of $y.$
\begin{prop}\label{prop13}
Let $y$ be the solution of \eqref{71}. Then as $r\goto\infty,$ $y$ has the following asymptotic expansion
$$y(r)=\td{\C}-\frac{1}{\C^2}\sqrt[k]{\frac{n-k}{n}}\frac{1}{r}+O\left(\frac{1}{r^2}\right).$$
\end{prop}
\begin{proof}
By Proposition \ref{prop12} we may assume
\be\label{define z}
y(r)=\td{\C}-\frac{z}{r}.
\ee Then we have,
\begin{eqnarray}\label{72}
\sqrt{1-y^2}-\frac{1}{\C}=\frac{1-\frac{1}{\C^2}-y^2}{\sqrt{1-y^2}+\frac{1}{\C}}=\frac{z}{r}A(r),
\text{ where }  A(r)=\frac{\sqrt{1-\frac{1}{\C^2}}+y}{\sqrt{1-y^2}+\frac{1}{\C}}.
\end{eqnarray}
Differentiating \eqref{define z}, then substituting it into \eqref{71}, we get
\be\label{rs1}
\frac{k}{n}\frac{y^{k-1}}{1-y^2}\left(-\frac{z'}{r^k}+\frac{z}{r^{k+1}}\right)+\frac{n-k}{n}\frac{y^k}{r^k}=\C^k\left(\sqrt{1-y^2}-\frac{1}{\C}\right)^k.
\ee
By \eqref{72}, \eqref{rs1} can be simplified as
$$\frac{k}{n}\frac{y^{k-1}}{1-y^2}\left(-z'+\frac{z}{r}\right)+\frac{n-k}{n}y^k=\C^kz^kA^k(r).$$
Thus, we obtain
\begin{eqnarray}\label{73}
z'=-B(r)z^k+C(r),
\end{eqnarray}
where
\be\label{BC}
B(r)=\C^k\frac{n}{k}\frac{1-y^2}{y^{k-1}}A^k(r)\,\,\text{and}\,\,  C(r)=\frac{z}{r}+\frac{n-k}{k}y(1-y^2).
\ee
Applying Proposition \ref{prop12} we can see that
$$\lim_{r\rightarrow +\infty}B(r)=\frac{n}{k}\C^{2k-2}\td{\C}\,\,\text{and}\,\,
\lim_{r\rightarrow +\infty}C(r)=\frac{n-k}{k}\frac{1}{\C^2}\td{\C}.$$
Here, we have used $\lim\limits_{r\goto\infty}\frac{z}{r}=0,$ which is a direct consequence of Proposition \ref{prop12}.
Next Lemma is a generalization of Proposition A.2 in \cite{Ba2020}.
\begin{lemm}
\label{A.2}
Assume $z:(0,+\infty)\rightarrow \mathbb{R}$ is a positive solution of the equation
$$z'=-A(r)z^k+B(r),$$ where $A,B: (0, \infty)\goto\R$ are continuous functions such that
$$\lim_{r\rightarrow+\infty}A(r)=A_0>0, \lim_{r\rightarrow+\infty}B(r)=B_0>0.$$ Then
$$\lim_{r\rightarrow+\infty}z(r)=\sqrt[k]{\frac{B_0}{A_0}}.$$
\end{lemm}
\begin{proof}
In order to prove this Lemma, we only need to prove
\begin{claim}
Assume $z:(0,+\infty)\rightarrow \mathbb{R}$ is a positive solution of the equation
\[z'=A_0z^k+B_0,\]
with $A_0<0, $ $B_0>0$ being constants. Then
\[\lim\limits_{r\goto\infty}z(r)=\lt(-\frac{B_0}{A_0}\rt)^{1/k}.\]
\end{claim}
If this claim is true, following the same argument as Proposition A.2 in \cite{Ba2020}, we can prove Lemma \ref{A.2}.
We will prove this claim below.

Without loss of generality, let's consider the positive solution of equaiton
\be\label{rs2}
z'=B-z^k
\ee
instead. We will show that
\be\label{rs3}
\lim\limits_{r\goto\infty}z(r)=B^{1/k}.
\ee
First, since $z$ is a positive solution of \eqref{rs2}, let's assume $0<z(r_0)=z_0<B^{1/k}$ then we have $z_0< z(r)<B^{1/k}$ on $(r_0, \infty).$
Denote $z_1=B^{1/k}$ we get
\[z^k-B=(z-z_1)(z^{k-1}+z^{k-2}z_1+\cdots+z_1^{k-1}).\]
Therefore \eqref{rs2} can be written as
\be\label{rs4}
-dr=\lt[\frac{A_1}{z-z_1}+\frac{Q_{k-2}(z)}{z^{k-1}+z^{k-2}z_1+\cdots+z_1^{k-1}}\rt]dz,
\ee
where $A_1=\frac{1}{k}z_1^{1-k}$ and $Q_{k-2}(z)$ is a polynomial of degree $k-2.$ It's easy to see that
$Q_{k-2}(z)=-A_1z^{k-2}+Q(k-3)(z)$ and $Q_{k-3}(z)$ is a polynomial of degree $k-3.$
Integrating \eqref{rs4} from $r_0$ to $r$ yields
\begin{eqnarray}\label{rs5}
-r+r_0&=&A_1\ln\lt|\frac{z(r)-z_1}{z_0-z_1}\rt|-\int_{z_0}^{z(r)}\frac{A_1z^{k-2}}{z^{k-1}+z^{k-2}z_1+\cdots+z_1^{k-1}}dz\\
&&+\int_{z_0}^{z(r)}\frac{Q_{k-3}(z)}{z^{k-1}+z^{k-2}z_1+\cdots+z_1^{k-1}}dz\nonumber.
\end{eqnarray}
Notice that as $r\goto\infty$ the left hand side of \eqref{rs5} goes to $-\infty,$ while
\[-\int_{z_0}^{z(r)}\frac{A_1z^{k-2}}{z^{k-1}+z^{k-2}z_1+\cdots+z_1^{k-1}}dz\geq-A_1\ln\lt|\frac{z_1}{z_0}\rt|,\]
and
\[\lt|\int_{z_0}^{z(r)}\frac{Q_{k-3}(z)}{z^{k-1}+z^{k-2}z_1+\cdots+z_1^{k-1}}dz\rt|\]
is bounded. Therefore, $\lim\limits_{r\goto\infty}z(r)=z_1=B^{1/k}.$ Similarly, we can prove the case when $z(r_0)=z_0>z_1.$
\end{proof}

From Lemma \ref{A.2} and equation \eqref{73}, we conclude
$$\lim_{r\rightarrow +\infty}z(r)=\frac{1}{\C^2}\sqrt[k]{\frac{n-k}{n}}.$$
We further assume
$$z(r)=\frac{1}{\C^2}\sqrt[k]{\frac{n-k}{n}}+\frac{w(r)}{r}.$$
Inserting it into \eqref{73}, we get
$$w'=-D(r)w+F(r),$$ where $$D(r)=B(r)\sum_{i=1}^k\lt(\begin{matrix}k\\ i\end{matrix}\rt)\left(\frac{1}{\C^2}\sqrt[k]{\frac{n-k}{n}}\right)^{k-i}\left(\frac{w}{r}\right)^{i-1},\ \
F(r)=r\left(C(r)-\frac{B(r)}{\C^{2k}}\frac{n-k}{n}\right)+\frac{w}{r}.$$
Notice that $\lim\limits_{r\rightarrow +\infty}\frac{w}{r}=0$ and $D(r)$ has a uniform positive lower bound. In the following,
we want to find a positive upper bound for $F(r)$. Using the expressions \eqref{BC} of $B(r),C(r)$, we obtain
\begin{eqnarray}
F(r)&=&\frac{w}{r}+z+\frac{n-k}{k}\frac{1-y^2}{y^{k-1}}r\left[y^k-\left(\frac{A(r)}{\C}\right)^k\right]\\
&=&\frac{w}{r}+z+\frac{n-k}{k}\frac{1-y^2}{y^{k-1}}r\left(y-\frac{A(r)}{\C}\right)\sum_{i=1}^{k}y^{k-i}\left(\frac{A(r)}{\C}\right)^{i-1}\nonumber.
\end{eqnarray}
Therefore, we only need to show $r(y-A(r)/\C)$ is bounded as $r\goto\infty.$ By \eqref{72}, we have
\be\label{75}
r\left(y-\frac{A(r)}{\C}\right)=r\left(y-\frac{1}{\C}\frac{\sqrt{1-\frac{1}{\C^2}}+y}{\sqrt{1-y^2}+\frac{1}{\C}}\right)
=\frac{r\left(y\sqrt{1-y^2}-\frac{1}{\C}\sqrt{1-\frac{1}{\C^2}}\right)}{\sqrt{1-y^2}+\frac{1}{\C}}.
\ee
Combining \eqref {75} with the expression of $y$ and \eqref{72}, we can derive
\be\label{76}
\begin{aligned}
y\sqrt{1-y^2}-\frac{1}{\C}\sqrt{1-\frac{1}{\C^2}}&=\left(\sqrt{1-\frac{1}{\C^2}}-\frac{z}{r}\right)\left(\frac{1}{\C}+\frac{zA(r)}{r}\right)
-\frac{1}{\C}\sqrt{1-\frac{1}{\C^2}}\\
&=\frac{z}{r}\left(-\frac{1}{\C}+A(r)\sqrt{1-\frac{1}{\C^2}}\right)-\frac{z^2A(r)}{r^2}.
\end{aligned}
\ee
 From \eqref{75}, \eqref{76}, and Lemma \ref{A.2} we conclude that $r(y-A(r)/\C)$ is uniformly bounded from above. Thus, $F(r)$ has an uniform upper bound. Applying Proposition A.3 in \cite{Ba2020}, we obtain a uniform upper bound for $w.$ This completes the proof.
\end{proof}
It's not hard to see that Theorem \ref{theo11} follows from Proposition \ref{prop12} and Proposition \ref{prop13}.

\section{The existence results}
\label{existence for soliton}
In this section we will prove Theorem \ref{theo3}. First, we want to prove the following existence Theorem.
\begin{prop}
\label{es-th1}
Suppose $\varphi$ is a $C^2$ function defined on $\mathbb{S}^{n-1}_{\td{\C}}:=\{x\in\R^n| |x|=\td{\C}\},$ where $\td{\C}=\sqrt{1-\lt(\frac{1}{\C}\rt)^2}.$
There exists a unique, strictly convex solution $u:\mathbb{R}^n\rightarrow\mathbb{R}$ of \eqref{soliton} such that as $|x|\goto\infty,$
\begin{eqnarray}\label{asymptotic behavior}
u(x)\goto\tilde{\C}|x|-\frac{1}{\C^2}\sqrt[k]{\frac{n-k}{n}}\log|x|+\varphi\left(\tilde{\C}\frac{x}{|x|}\right).
\end{eqnarray}
\end{prop}
\par
\subsection{Constructing barriers} We first construct the barrier functions of equation \eqref{soliton}.
Following the ideas of \cite{SX16, T}, we denote the radial solution of \eqref{soliton} by $z_0^k(|x|),$ whose asymptotic expansion satisfies
\eqref{asymptotic} with $c_0=0$.
Let $$p_i(\tilde{\C}y)=D\varphi(\tilde{\C}y)+(-1)^{i+1}2M\tilde{\C}y,\,\,i=1,2$$
for any $y\in\mathbb{S}^{n-1}$.
Set,
$$z_i^k(x,y)=\varphi(\tilde{\C}y)-p_i(\tilde{\C}y)\cdot \tilde{\C}y+z_0^k(|x+p_i(\tilde{\C}y)|),\,\forall x\in\mathbb{R}^n, y\in\mathbb{S}^{n-1}.$$
Then,
$$q_1^k(x)=\sup_{y\in\mathbb{S}^{n-1}}z_1^k(x,y)$$
is a subsolution of \eqref{soliton} and
$$q_2^k=\inf_{y\in\mathbb{S}^{n-1}}z_2^k(x,y)$$ is a supersolution of \eqref{soliton}.
Moreover, $q_1^k(x)\leq q_2^k(x)$ and when $|x|\rightarrow+\infty$, we have
$$q_i^k(x)\rightarrow \tilde{\C}|x|-\frac{1}{\C^2}\sqrt[k]{\frac{n-k}{n}}\log|x|+\varphi\left(\tilde{\C}\frac{x}{|x|}\right),\,\,i=1, 2.$$
\par
\subsection{The Dirichlet problem}
First, let's solve equation \eqref{soliton} for the case when $k=n$. For any $t>\min_{\mathbb{R}^n}q_2^n$, we let
$$\p\Omega_t=\{x\in\mathbb{R}^n|q_1^n(x)<t<q_2^n(x)\},$$
and $\Omega_t$ be a smooth, strictly convex domain in $\R^n.$
Consider the following Dirichlet problem:
\be\label{D4}
\left\{
\begin{aligned}
\sigma_n^{1/n}(\kappa(\M_{u_t}))&=\C+\langle \nu, E\rangle\,\, &\text{in $\Omega_t$}\\
u_t&=t\,\,&\text{on $\p \Omega_t$}
\end{aligned}
\right..
\ee
By a small modification of \cite{Del}, we know that there exists a unique solution $u_t$ of \eqref{D4}.
Then, applying the local $C^1$, $C^2$ estimates obtained in \cite{BS} we conclude that,
there exists a subsequence $\{u_{t_i}\}_{i=1}^\infty\,\,(t_i\goto\infty$ as $i\goto\infty),$ that converges to an entire,
strictly convex solution $u$ of \eqref{soliton} for $k=n.$
Moreover, it's easy to see that $u(x)$ satisfies the desired asymptotic behavior as $|x|\goto\infty$.
From now on, we will denote this solution by $u^n.$ We will also denote the Legendre transform of $u^n$ by $u^{n*}$.

Next, we consider the case when $k<n$. We denote the legendre transform of $z_0^k$ by $(z_0^k)^*$, that is,
$$(z^k_0)^*(\tau)=r\cdot \frac{\p z^k_0}{\p r}-z^k_0(r),\text{ where } \tau=\frac{\p z^k_0}{\p r}.$$
Using the asymptotic expansion of $z_0$ derived in Section \ref{rs}, we know
$$(z^k_0)^*(\tau)=\frac{1}{\C^2}\sqrt[k]{\frac{n-k}{n}}(\log r-1)+O\left(\frac{1}{r}\right).$$
We denote its principal part:
$$(\tilde{z}^k_0)^*(\tau)=\frac{1}{\C^2}\sqrt[k]{\frac{n-k}{n}}(\log r(\tau)-1),$$
it is clear that $(\tilde{z}^k_0)^*$ is unbounded in $B_{\td{\C}}$.

To make sure our solution is convex, we consider the dual Dirichelt problem
on $B_\tau$ for any $\tau<\tilde{\C}$,
\be\label{D5}
\left\{
\begin{aligned}
\hF(\w\gas_{ik}\us_{kl}\gas_{lj})&=\frac{\binom{n}{k}^{-1/k}}{\C-\frac{1}{\sqrt{1-|\xi|^2}}}\,\,\text{in $B_{\tau}$},\\
\us&=u^{n*}+(z^k_0)^*-(z^n_0)^*\,\,\text{on $\p B_{\tau}$}.
\end{aligned}
\right.
\ee
Here, we have $\w=\sqrt{1-|\xi|^2},$ $\gas_{ij}=\delta_{ij}-\frac{\xi_i\xi_j}{1+\w},$ $\us_{kl}=\frac{\p^2u}{\p\xi_k\p\xi_l},$  $\hF(\w\gas_{ik}\us_{kl}\gas_{lj})$ $=\lt(\frac{\s_n}{\s_{n-k}}(\ka^*[\w\gas_{ik}\us_{kl}\gas_{lj}])\rt)^{1/k},$ and
$\ka^*[\w\gas_{ik}\us_{kl}\gas_{lj}]=(\ka^*_1, \cdots, \ka^*_n)$ are the eigenvalues of the matrix $(\w\gas_{ik}\us_{kl}\gas_{lj}).$
The solvability of \eqref{D5} has been established in Section \ref{Dirichlet}. Therefore, by standard PDE theorems, in order to prove
Proposition \ref{es-th1} we only need to obtain
local $C^1$  and local $C^2$ estimates for the translating soliton equation \eqref{soliton}. In order to do so, we will need the following Lemma.

\begin{lemm}\label{lem7.2.1} Let $u^{\tau*}$ be a solution to equation \eqref{D5} and $u^\tau$ be the Legendre transform of $u^{\tau*}.$ Then,
for any $x\in Du^{\tau*}(B_\tau),$ we have
 $q_1^k(x)\leq u^\tau(x)\leq q_2^k(x).$
\end{lemm}
\begin{proof} Without causing confusion we shall drop the superscript $\tau$ in the proof.
We only need to prove that
$$z_1^k(x,y)\leq u(x)\leq z_2^k(x,y),$$ for any $x\in Du^{\tau*}(B_\tau)$ and $y\in\mathbb{S}^{n-1}$. This is equivalent to prove
$$(z_2^k)^*(\xi,y)\leq u^*(\xi)\leq (z_1^k)^*(\xi,y),$$
for any $\xi\in B_\tau$ and $y\in\mathbb{S}^{n-1}.$  Since we have
\begin{eqnarray}
(z_i^k)^*(\xi,y)&=&(z_0^k)^*(|\xi|)-p_i(\tilde{\C}y)\cdot\xi-\varphi(\tilde{\C}y)+p_i(\tilde{\C}y)\cdot \tilde{\C}y\\
&=&(z_0^k)^*(|\xi|)-(z_0^n)^*(|\xi|)+(z_i^n)^*(\xi,y),\nonumber
\end{eqnarray}
and
$$(z_2^n)^*(\xi,y)< u^{n*}(\xi)< (z_1^n)^*(\xi,y),$$ we obtain on $\p B_{\tau}$,
$$(z_2^k)^*(\xi,y)\leq u^*(\xi)\leq (z_1^k)^*(\xi,y).$$
By comparison principle, we finish the proof.
\end{proof}
\par
\subsection{Local $C^1$ and $C^2$ estimates}
Similar to Lemma
\ref{lc1lem1}, we have the following local $C^1$ estimate Lemma for translating solitons.
\begin{lemm}
\label{lemm81}
Let $\Omega\subset \R^n$ be a bounded open set. Let $u, \bar{u}, \Psi:\Omega\rightarrow\mathbb{R}^n$ be strictly $\C$-spacelike, i.e.
$$|Du|,|D\uu|,|D\Psi|<\td{\C}.$$
Assume that $u$ is strictly convex and $u\leq\bar{u}$ in $\Omega.$ Also
assume that near $\partial\Omega,$ we have $\Psi>\bar{u}.$
Consider the set, where $u>\Psi.$ For every $x$ in that set, we have the following gradient estimate for $u:$
\[\frac{1}{\sqrt{\td{\C}^2-|Du|^2}}\leq\frac{1}{u(x)-\Psi(x)}\cdot\sup\limits_{\{u>\Psi\}}\frac{\bar{u}-\Psi}{\sqrt{\td{\C}^2-|D\psi|^2}}.\]
\end{lemm}
Since the proof is the same as the proof of Lemma 5.1 in \cite{BS}, we skip it here.

We now construct $\Psi.$ Following the argument in Section 4 of \cite{Ba2020}, let
$$\Psi(x)=-A_0+\tilde{\C}\sqrt{1+|x|^2}.$$ It is clear that when $|x|$ sufficiently large we have $\Psi(x)>q_2(x)$.
On the other hand, for any compact set $\K\subset\R^n$, we can always choose $A_0$ sufficiently large such that $\Psi(x)<q_1(x)$ in $\K$.
Applying Lemma \ref{lemm81} we obtain that for any $\K\subset \R^n$ and any strictly convex function
$q_1(x)<u(x)<q_2(x)$ satisfying \eqref{soliton}, whose domain of definition contains $\K,$ there exists a local $C^1$ bound $C_\K$
for $u(x)$ in $\K$ that is only depending on $\K$.

Using the idea of \cite{WX}, we can prove the following Pogorelov type local $C^2$ estimate
for translating solitons.
\begin{lemm}
\label{lc2lem1}
Let $u$ be the solution of \eqref{soliton} defined on $\Omega$.
For any given $s>\min\limits_{\R^n}u(x)+1,$ suppose $u|_{\T\Omega}>s.$ Let $\kappa_{\max}(x)$
be the largest principal curvature of $\M_{u}=\{(x, u(x))|x\in\Omega\}$ at $x$.
Then, we have
\[\max\limits_{\M_{u}}(s-u)\kappa_{\max}\leq C_1.\]
Here, $C_1$ only depends on the local $C^1$ estimate of $u$. More specifically, $C_1$ depends on the lower bound of
$\C+\langle \nu, E\rangle.$
\end{lemm}

Following the argument in Section \ref{prescribed curvature}, we complete the proof of Proposition \ref{es-th1}.
\par
\subsection{Proof of Theorem \ref{theo3}}
In this subsection, we will prove that the hypersurface $\M_u$ constructed in Proposition \ref{es-th1} has bounded principal curvatures. This completes the proof of Theorem \ref{theo3}. For our convenience, in the following, we will drop the superscript $k,$ and the updated configuration $z_0^k$
now becomes $z_0.$

Suppose $u$ is a strictly convex solution of \eqref{soliton} and $u^*$ is the Legendre transform of $u$. Then $\us$ satisfies
\be\label{pfth4.2}
\hF(\w\gas_{ik}\us_{kl}\gas_{lj})=\frac{\binom{n}{k}^{-1/k}}{\C-\frac{1}{\sqrt{1-|\xi|^2}}}\,\,\text{in $B_{\td{\C}}.$}
\ee
We also denote the Legendre transform of $z_0$ by $z_0^*$, that is,
$$z^*_0(\tau)=r\cdot \frac{\p z_0}{\p r}-z_0(r),\text{ where } \tau=\frac{\p z_0}{\p r}.$$
Using the asymptotic expansion of $z_0$ derived in Section \ref{rs}, we know
$$z^*_0(\tau)=\frac{1}{\C^2}\sqrt[k]{\frac{n-k}{n}}(\log r-1)+O\left(\frac{1}{r}\right).$$
We denote its principal part as
$$\tilde{z}^*_0(\tau)=\frac{1}{\C^2}\sqrt[k]{\frac{n-k}{n}}(\log r(\tau)-1),$$
it is clear that $\tilde{z}^*_0(\tau)$ is unbounded in $B_{\td{\C}}$.
\begin{lemm}
\label{sub8.4lem1}
Let $u^*$ and $\td{z}^*_0$ be defined as above. Then we have,
\begin{eqnarray}\label{85}
\lim_{\xi\rightarrow \xi_0}(\us(\xi)-\tilde{z}^*_0(|\xi|))=-\varphi(\xi_0), \text{ for any } \xi_0\in \p B_{\tilde{\C}}, \xi\in B_{\tilde{\C}}.
\end{eqnarray}
\end{lemm}
\begin{proof}We will use the auxiliary functions $z_i(x, y),\,\, i=1, 2,$ constructed in Subsection 7.1.
It's easy to see that
$$z_1(x,y)< u(x)< z_2(x,y), \text{ for any } x\in\mathbb{R}^n, y\in \mathbb{S}^{n-1}.$$
By the strictly convexity of $z_i(x, y)$ we have
\be\label{pfth4.1}
z_2^*(\xi,y)< \us(\xi)< z_1^*(\xi,y),\,\,\text{for any}\,\, \xi\in B_{\td{\C}},\,\,y\in\mathbb{S}^{n-1}.
\ee
Notice that
$$z_i^*(\xi,y)=z_0^*(|\xi|)-p_i(\tilde{\C}y)\cdot\xi-\varphi(\tilde{\C}y)+p_i(\tilde{\C}y)\cdot \tilde{\C}y.$$
Therefore, let $\tilde{\C}y=\xi_0$ and $\xi\rightarrow\xi_0$, we get
$$z_i(\xi, \tilde{\C}^{-1}\xi_0)-z_0^*(|\xi|)\rightarrow -\varphi(\xi_0).$$
This together with \eqref{pfth4.1} yields \eqref{85}.
\end{proof}

Now we let
$$\T=\xi_i\frac{\p}{\p \xi_j}-\xi_j\frac{\p}{\p \xi_i}$$ be the angular derivative. Similar to Section 10 in \cite{RWX}, we obtain
following Lemmas.
\begin{lemm}
\label{lemm17} Let $\us$ be the solution of equation \eqref{pfth4.2}.
Then, $|\T \us|$ are bounded above by a constant depends on $|\vp|_{C^1}$ and $\T^2\us$ are bounded above by a constant depends on $|\vp|_{C^2}$.
\end{lemm}
\begin{proof}
Notice that $\p |\xi|^2=0$, we have the angular derivative of the right hand side of equation \eqref{pfth4.2}
is zero. Therefore, following the proof of Lemma 29 and 30 in \cite{RWX}, we have
$$F^{ij}\w\gas_{ik}(\p(\us-\tilde{z}^*_0))_{kl}\gas_{lj}=0, F^{ij}\w\gas_{ik}(\p^2(\us-\tilde{z}^*_0))_{kl}\gas_{lj}\geq 0.$$
In view of \eqref{85} and the maximum principle, we obtain the desired estimates.
\end{proof}
We further have
\begin{lemm}
\label{cvlem1.3}
Let $\us$ be the solution of equation \eqref{pfth4.2}.
There is a positive constant $b$ such that
\[\sqrt{\tilde{\C}^2-|\xi|^2}\lt|\T^2\us\rt|<b.\]
\end{lemm}
\begin{proof}
We consider $\us-\tilde{z}_0^*$, which has $C^0$ bound on $B_{\tilde{\C}}$. Since $\T^2 \us=\T^2 (\us-\tilde{z}_0^*)$, the rest of the proof is same as the one of Lemma 5.3 in \cite{LA}.
\end{proof}

\begin{lemm}
\label{lemm19}
Suppose $a_0<r<\td{\C}$ for some $a_0\in (0, \td{\C}),$ and $\mathbb{S}^{n-1}(r)=\{\xi\in\R^n|\sum\xi_i^2=r^2\}.$
For any point $\hat{\xi}\in\mathbb{S}^{n-1}(r),$ there is a function
\[\uus_0=z_0^*+b_1\xi_1+ \cdots +b_n\xi_n+b\]
such that
\[\uus_0(\hat{\xi})=\us(\hat{\xi}),\]
and\[\uus_0(\hat{\xi})>\us(\xi),\,\,\mbox{for any $\xi\in\mathbb{S}^{n-1}(r)\setminus\{\hat{\xi}\}$}.\]
Here $b_1, \cdots, b_n$ are constants depending on $\hat{\xi},$ and $b$ is a positive constant independent of $\hat{\xi}$
and $r.$
\end{lemm}
\begin{proof}
The proof is almost the same as the proof of Lemma 5.4 in \cite{LA}. We only need to replace
$u,\uu,-\bar{k}\sqrt{1-|x|^2}$ by $\us-\tilde{z}_0^*, \uus_0-\tilde{z}_0^*$ and $z_0^*-\tilde{z}_0^*$ in Li's proof.
\end{proof}

Similarly, we can prove the following Lemma analogous to Lemma 5.5 in \cite{LA}.
\begin{lemm}
\label{lemm20}
Suppose $a_0<r<\td{\C}$ for some $a_0\in (0, \td{\C}),$ and $\mathbb{S}^{n-1}(r)=\{\xi\in\R^n|\sum\xi_i^2=r^2\}.$
For any point $\hat{\xi}\in \mathbb{S}^{n-1}(r),$ there is a function
\[\lus_0=z_0^*+a_1\xi_1+ \cdots +a_n\xi_n-a\]
such that
\[\lus_0(\hat{\xi})=\us(\hat{\xi}),\]
and\[\lus_0(\hat{\xi})<\us(\xi),\,\,\mbox{for any $\xi\in\mathbb{S}^{n-1}(r)\setminus\{\hat{\xi}\}$}.\]
Here $a_1, \cdots, a_n, a$ are constants depending on $\hat{\xi},$ and $a>0, a\sqrt{\tilde{\C}^2-|\hat{\xi}|^2}<C_1$, where $C_1$ is a positive constant only depending on $|\varphi|_{C^2}$.
\end{lemm}

Using Lemma \ref{lemm19} and Lemma \ref{lemm20} we can show
\begin{lemm}
\label{lemm21}
Let $u$ be the solution of equation \eqref{soliton} and $u^*$ be the Legendre transform of $u$.
There are positive constants $d_2>d_1$ such that
\be\label{cv1.17}
0<d_1\leq u(\tilde{\C}^2-|Du|^2)\leq d_2.
\ee
Here $d_2$  depends on the $|u|_{C^0(\Omega)}$ and $\Omega=\{x\in\mathbb{R}^n; |Du|\leq a_0\}.$
\end{lemm}
\begin{proof}
We modify the proof of Li \cite{LA}. We first consider the lower bound. For any $\hat{\xi}\in\mathbb{S}^{n-1}(r)$, using  Lemma \ref{lemm19}, we have
$$\us(\hat{\xi})=\uus_0(\hat{\xi}),\text{ and } \us(\xi)<\uus_0(\xi) \text{ for } \xi\in \mathbb{S}^{n-1}(r)\setminus\{\hat{\xi}\}.$$
Thus, using $\uus_0$ is a supersolution, we get $\us(\xi)<\uus_0(\xi)$ in $B_{r}$. Therefore, at $\hat{\xi}$, we get
$$u(\hat{x})=\hat{\xi}\cdot Du^*-u^*>\hat{\xi}\cdot D\uus_0-\uus_0=z_0(\hat{r})-b,$$
where we assume $\hat{x}=Du^*(\hat{\xi})$ and $z_0'(\hat{r}):=\frac{\p z_0}{\p r}(\hat{r})=|\hat{\xi}|$. Thus, at $\hat{x}$, we have
\begin{eqnarray}\label{88}
u(\tilde{\C}^2-|Du|^2)>z_0(\hat{r})(\tilde{\C}^2-|z_0'(\hat{r})|^2)-b(\tilde{\C}^2-|\hat{\xi}|^2).
\end{eqnarray}
Using the asymptotic behavior of $z_0$, we have
\begin{eqnarray}
&& z_0\lt(\tilde{\C}^2-|z'_0|^2\rt)\nonumber\\
&=&\left[\tilde{\C}r-\frac{1}{\C^2}\sqrt[k]{\frac{n-k}{n}}\log r+O\left(\frac{1}{r}\right)\right]
\left[\tilde{\C}^2-\left(\tilde{\C}-\frac{1}{\C^2}\sqrt[k]{\frac{n-k}{n}}\frac{1}{r}+O\left(\frac{1}{r^2}\right)\right)^2\right]
\nonumber\\
&=&2\frac{\tilde{\C^2}}{\C^2}\sqrt[k]{\frac{n-k}{n}}+o(1)\nonumber
\end{eqnarray}
We denote $$2c_0=2\frac{\tilde{\C^2}}{\C^2}\sqrt[k]{\frac{n-k}{n}}.$$ Therefore, by \eqref{88}, we obtain
$$u(\tilde{\C}^2-|Du|^2)>\frac{c_0}{2},$$ for $r$ being sufficiently close to $\td{\C}$, which we may assume $r>a_0$. For $r<a_0$, without loss of generality, we can assume $u\geq 1$. Therefore
$$u(\tilde{\C}^2-|\hat{\xi}|^2)\geq \tilde{\C}^2-a_0^2.$$ Thus, we obtain the uniform lower bound. For the upper bound. Applying a similar argument, for $r$ being sufficiently close to $\td{\C}$ , which we will still assume $r\geq a_0$, we have
\begin{eqnarray}
u(\tilde{\C}^2-|Du|^2)<z_0(\hat{r})(\tilde{\C}^2-|z_0'(\hat{r})|^2)+a(\tilde{\C}^2-|\hat{\xi}|^2)\leq 3c_0+C_1\tilde{\C}.\nonumber
\end{eqnarray}
We obtain a uniform upper bound.
\end{proof}
Finally, we are ready to adapt the ideas in \cite{RWX, LA} to estimate the principal curvatures of $\M_u.$
\begin{prop}
\label{ubprop1}
Let $u$ be the solution of equation \eqref{soliton}.
Then the hypersurface $\M_u=\{(x, u(x)) |\,x\in\R^n\}$ has bounded principal curvatures.
\end{prop}
\begin{proof}
We will establish a Pogorelov type interior estimate.
For any $s>0,$ consider
$$\phi=e^{-\frac{s}{s-u}}[u(\C+\langle\nu,E\rangle )]^{-N}P_m^{1/m},$$ where $P_m=\sum\limits_{j}\ka_j^m$ and $m, N>0$ are constants to be determined later. Without loss of generality, we also assume $u\geq 1$ in $\R^n$.
 It's easy to see that $\phi$ achieves its local maximum at an interior point of
 $U_s=\{x\in\R^n |\, u(x)<s\},$ we will assume this point is $x_0.$ We can choose a local normal coordinate
 $\{\tau_1, \cdots, \tau_n\}$ such that at $x_0,$
 $h_{ij}=\la_i\delta_{ij}$ and $\la_1\geq\la_2\geq\cdots\geq\la_n.$

Differentiating $\log\phi$ at $x_0$ we get,
\begin{equation}\label{12.2}
\frac{\phi_i}{\phi}=\frac{\dsum_j\kappa_j^{m-1}h_{jji}}{P_m}-N\frac{h_{ii}\langle \tau_i,E\rangle}{\C+\langle\nu,E\rangle}-N\frac{u_i}{u}-\frac{ su_i}{(s-u)^2}=0,
\end{equation}
and
\begin{eqnarray}\label{ub1.0}
&&\frac{\phi_{ii}}{\phi}-\frac{\phi_i^2}{\phi^2}\\
&=&\frac{1}{P_m}\lt[\dsum_j\kappa_j^{m-1}h_{jjii}+(m-1)\dsum_j\kappa_j^{m-2}h_{jji}^2
+\dsum_{p\neq q}\dfrac{\kappa_p^{m-1}-\kappa_q^{m-1}}{\kappa_p-\kappa_q}h_{pqi}^2\rt] \nonumber\\
&&-\dfrac{m}{P_m^2}\lt(\dsum_j\kappa_j^{m-1}h_{jji}\rt)^2 -N\sum_lh_{ili}\frac{\langle \tau_l,E\rangle}{\C+\langle\nu,E\rangle}
+Nh_{ii}^2\frac{-\langle \nu, E\rangle}{\C+\langle \nu, E\rangle}\nonumber\\
&&+Nh_{ii}^2\frac{u_i^2}{(\C+\langle\nu,E\rangle)^2}
+N\frac{h_{ii}\langle\nu,E\rangle}{u}+N\frac{u_i^2}{u^2}+ s\frac{h_{ii}\langle \nu,E\rangle}{(s-u)^2}-2s\frac{u_i^2}{(s-u)^3}\leq 0.\nonumber
\end{eqnarray}
By equation \eqref{soliton}, we derive
$$\sigma_k^{ii}h_{iij}=\binom{n}{k}k(\C+\langle\nu,E\rangle)^{k-1}(-h_{jj}u_j),$$
and
\begin{eqnarray}
\sigma_k^{ii}h_{iijj}&=&-\sigma_k^{pq,rs}h_{pqj}h_{rsj}+\binom{n}{k}k(k-1)(\C+\langle \nu, E\rangle)^{k-2}h_{jj}^2u_j^2\\ &&+\binom{n}{k}k(\C+\langle\nu,E\rangle)^{k-1}\lt(-\sum_lh_{jjl}u_l+h_{jj}^2\langle\nu,E\rangle\rt)\nonumber\\
&\geq&-\sigma_k^{pq,rs}h_{pqj}h_{rsj}+\binom{n}{k}k(\C+\langle\nu,E\rangle)^{k-1}\lt(-\sum_lh_{jjl}u_l\rt)\nonumber\\
&&-K_0(\C+\langle\nu,E\rangle)^{k-1}\kappa_1^2\nonumber,
\end{eqnarray}
where $K_0=K_0(n, k, \C)>0$ is a constant depending on $n, k$ and $\C.$
Recall that in Minkowski space we have $$h_{jjii}=h_{iijj}+h^2_{ii}h_{jj}-h_{ii}h^2_{jj}.$$
Thus,
\be\label{pfth4.3}
\sigma_k^{ii}h_{jjii}=\sigma_k^{ii}h_{iijj}+\sigma_k^{ii}h_{ii}^2h_{jj}-\sigma_k^{ii}h_{ii}h_{jj}^2\geq\sigma_k^{ii}h_{iijj}- k\binom{n}{k}(\C+\langle\nu,E\rangle)^kh_{jj}^2.
\ee
Combining \eqref{pfth4.3} with \eqref{ub1.0} we obtain
\be\label{ub1.2*}
\begin{aligned}
0&\geq \sigma_{k}^{ii}\frac{\phi_{ii}}{\phi}=\frac{\sigma_{k}^{ii}}{P_m}\lt[\dsum_j\kappa_j^{m-1}h_{jjii}+(m-1)\dsum_j\kappa_j^{m-2}h_{jji}^2
+\sum_{p\neq q}\dfrac{\kappa_p^{m-1}-\kappa_q^{m-1}}{\kappa_p-\kappa_q}h_{pqi}^2\rt]\\
&-\frac{m\sigma_{k}^{ii}}{P_m^2}\lt(\dsum_j\kappa_j^{m-1}h_{jji}\rt)^2 -N\sigma_{k}^{ii}\sum_lh_{ili}\frac{\langle \tau_l,E\rangle}{(\C+\langle\nu,E\rangle)}\\
&+N\sigma_{k}^{ii}h_{ii}^2\frac{-\langle \nu, E\rangle}{\C+\langle \nu, E\rangle}
+N\sigma_{k}^{ii}h_{ii}^2\frac{u_i^2}{(\C+\langle\nu,E\rangle)^2}\\
&+N\sigma_{k}^{ii}\frac{h_{ii}\langle\nu,E\rangle}{u}+N\sigma_{k}^{ii}\frac{u_i^2}{u^2}+ s\frac{\sigma_{k}^{ii}h_{ii}\langle \nu,E\rangle}{(s-u)^2}-2 s\frac{\sigma_{k}^{ii}u_i^2}{(s-u)^3}\\
&\geq -K_0(\C+\langle\nu,E\rangle)^{k-1}\kappa_1+\sum_i(A_i+B_i+C_i+D_i-E_i)\\
&+\lt(\begin{matrix}n\\ k\end{matrix}\rt)k(\C+\langle\nu,E\rangle)^{k-1}\frac{-\sum_{j,l}h_{jjl}\kappa_j^{m-1}u_l}{P_m}-Nk\lt(\begin{matrix}n\\ k\end{matrix}\rt)(\C+\langle\nu,E\rangle)^{k-2}\sum_l\kappa_lu_l^2\\
&+N\sigma_k^{ii}\kappa_i^2\frac{-\langle\nu,E\rangle}{\C+\langle\nu,E\rangle}+N\sigma_{k}^{ii}h_{ii}^2\frac{u_i^2}{(\C+\langle\nu,E\rangle)^2}\\
&+N\sigma_{k}^{ii}\frac{h_{ii}\langle\nu,E\rangle}{u}+N\sigma_{k}^{ii}\frac{u_i^2}{u^2}+ s\frac{\sigma_{k}^{ii}h_{ii}\langle \nu,E\rangle}{(s-u)^2}-2 s\frac{\sigma_{k}^{ii}u_i^2}{(s-u)^3}.
\end{aligned}
\ee
Here
$$A_i=\dfrac{\kappa_i^{m-1}}{P_m}\lt[K(\sigma_{k})_i^2-\dsum_{p,q}\sigma_{k}^{pp,qq}h_{ppi}h_{qqi}\rt],\,\,\mbox{for some constant $K>1$,}$$
$$B_i=\dfrac{2\kappa_j^{m-1}}{P_m}\dsum_j\sigma_{k}^{jj,ii}h_{jji}^2,$$
$$C_i=\dfrac{m-1}{P_m}\sigma_{k}^{ii}\dsum_j\kappa_j^{m-2}h_{jji}^2,$$
$$D_i=\dfrac{2\sigma_{k}^{jj}}{P_m}\dsum_{j\neq i}\dfrac{\kappa_j^{m-1}-\kappa_i^{m-1}}{\kappa_j-\kappa_i}h_{jji}^2,$$
and
$$E_i=\dfrac{m\sigma_{k}^{ii}}{P_m^2}\lt(\sum_j\kappa_j^{m-1}h_{jji}\rt)^2.$$
By Lemma 8, Lemma 9, and Corollary 10 in \cite{LRW} we can assume the following claim holds.
\begin{claim}
\label{cvlem1.900}
There exists two small positive constants $\delta$ and $\eta<1$. If $\kappa_k\leq \delta\kappa_1$,
we have
\be\label{cv1.240}
\sum_{i}A_i+B_i+C_i+D_i-\lt(1+\frac{\eta}{m}\rt)E_i\geq 0,
\ee
where $m>0$ is sufficiently large.
\end{claim}
If \eqref{cv1.240} doesn't hold, we would have $\kappa_k> \delta\kappa_1$. Since $\sigma_{k}\leq \lt(\begin{matrix}n\\ k\end{matrix}\rt)\C^k$, we get
$$\delta^{k-1}\kappa_1^k\leq \kappa_1\kappa_2\cdots\kappa_k\leq \sigma_k\leq \lt(\begin{matrix}n\\ k\end{matrix}\rt)\C^k.$$
This gives an upper bound for
$\la_1$ at $x_0$ directly, then we would be done. Therefore, we assume \eqref{cv1.240} holds.
Plugging \eqref{cv1.240} into \eqref{ub1.2*} yields,
\be
\begin{aligned}\label{neweq}
0&\geq-K_0(\C+\langle\nu,E\rangle)^{k-1}\kappa_1+\eta\frac{\sigma_{k}^{ii}}{P_m^2}
\lt(\sum_j\kappa_j^{m-1}h_{jji}\rt)^2\\
&-k \lt(\begin{matrix}n\\ k\end{matrix}\rt)(\C+\langle\nu,E\rangle)^{k-1}|\nabla u|^2\left(\frac{N}{u}+\frac{s}{(s-u)^2}\right)\\\
&+N\sigma_k^{ii}\kappa_i^2\frac{-\langle\nu,E\rangle}{\C+\langle\nu,E\rangle}+N\sigma_{k}^{ii}h_{ii}^2\frac{u_i^2}{(\C+\langle\nu,E\rangle)^2}\\
&+N\sigma_{k}^{ii}\frac{h_{ii}\langle\nu,E\rangle}{u}+N\sigma_{k}^{ii}\frac{u_i^2}{u^2}+ s\frac{\sigma_{k}^{ii}h_{ii}\langle \nu,E\rangle}{(s-u)^2}-2 s\frac{\sigma_{k}^{ii}u_i^2}{(s-u)^3}.
\end{aligned}
\ee
From equation \eqref{12.2} we obtain
\be\label{ub1.1}
\begin{aligned}
\left(\frac{\dsum_j\kappa_j^{m-1}h_{jji}}{P_m}\right)^2&=N^2\frac{\kappa_i^2u_i^2}{(\C+\langle\nu,E\rangle)^2}+N^2\frac{u_i^2}{u^2}+\frac{s^2u_i^2}{(s-u)^4}\\
&-2N^2\frac{\kappa_iu_i^2}{u(\C+\langle \nu,E\rangle)}-2N s\frac{\kappa_iu_i^2}{(\C+\langle\nu,E\rangle)(s-u)^2}+2N s\frac{u_i^2}{u(s-u)^2}.
\end{aligned}
\ee
Inserting \eqref{ub1.1} into \eqref{neweq}, we derive
\begin{eqnarray}\label{pfth4.4}
&&0\geq-K_0(\C+\langle\nu,E\rangle)^{k-1}\kappa_1
+\eta\frac{s^2\sigma_{k}^{ii}u_i^2}{(s-u)^4}+N(N\eta+1)\sigma_{k}^{ii}\kappa_i^2\frac{u_i^2}{(\C+\langle\nu,E\rangle)^2}\\
&&-2N^2\eta\frac{\sigma_{k}^{ii}\kappa_iu_i^2}{u(\C+\langle \nu,E\rangle)}-2N s\eta\frac{\sigma_{k}^{ii}\kappa_iu_i^2}{(\C+\langle\nu,E\rangle)(s-u)^2}+2N s\eta\frac{\sigma_{k}^{ii}u_i^2}{u(s-u)^2}\nonumber\\
&&+N\sigma_{k}^{ii}\frac{h_{ii}\langle\nu,E\rangle}{u}+N(\eta N+1)\sigma_{k}^{ii}\frac{u_i^2}{u^2}+ s\frac{\sigma_{k}^{ii}h_{ii}\langle \nu,E\rangle}{(s-u)^2}-2 s\frac{\sigma_{k}^{ii}u_i^2}{(s-u)^3}\nonumber\\
&&-k \lt(\begin{matrix}n\\ k\end{matrix}\rt)(\C+\langle\nu,E\rangle)^{k-1}|\nabla u|^2\left(\frac{N}{u}+\frac{s}{(s-u)^2}\right)+N\sigma_{k}^{ii}\kappa_i^2\frac{-\langle\nu,E\rangle}{\C+\langle\nu,E\rangle}\nonumber.
\end{eqnarray}
It's clear that
\be\label{pfth4.5}|\nabla u|=\frac{|Du|}{\sqrt{1-|Du|^2}}<-\langle\nu,E\rangle\leq \C.\ee
We also notice that for any $1\leq i\leq n,$ $\sigma_{k}^{ii}\kappa_i\leq \lt(\begin{matrix}n\\ k\end{matrix}\rt)\C^k$ (no summation). By a simple calculation we get, when $N>\frac{1}{\eta^2}$
\be\label{pfth4.6}\eta\frac{s^2\sigma_{k}^{ii}u_i^2}{(s-u)^4}+2Ns\eta\frac{\sigma_{k}^{ii}u_i^2}{u(s-u)^2}-2s\frac{\sigma_{k}^{ii}u_i^2}{(s-u)^3}\geq 0.\ee
Moreover, applying Lemma \ref{lemm21} we know there exists two positive constants $\td{d}_2>\td{d}_1>0$ such that
\be\label{pfth4.7}\td{d}_1\leq u\lt(\C+\lt<\nu, E\rt>\rt)\leq\td{d}_2.\ee
Therefore, for $N>\frac{1}{\eta^2}$ being sufficiently large, combining \eqref{pfth4.5}-\eqref{pfth4.7} with \eqref{pfth4.4} we have,
\begin{eqnarray}
0&\geq&-K_0(\C+\langle\nu,E\rangle)^{k-1}\kappa_1-\frac{2N^2}{\tilde{d}_1}|\nabla u|^2\sigma_k^{ii}\kappa_i-2N s\frac{|\nabla u|^2 \sigma_k^{ii}\kappa_i}{(\C+\langle\nu,E\rangle)(s-u)^2}\nonumber\\
&&-N\C\sigma_k^{ii}\kappa_i- \C\sigma_k^{ii}\kappa_i\frac{s}{(s-u)^2}-k\C^2\lt(\begin{matrix}n\\ k\end{matrix}\rt)(\C+\langle\nu,E\rangle)^{k-1}\frac{s}{(s-u)^2}\nonumber\\
&&-k \lt(\begin{matrix}n\\ k\end{matrix}\rt)\C^2(\C+\langle\nu,E\rangle)^{k-1}N+N\frac{c_0\sigma_k\kappa_1}{\C+\langle\nu,E\rangle}\nonumber.
\end{eqnarray}
It's easy to see that the above inequality
yields, at $x_0$
$$\kappa_1\leq K(N,\C,\tilde{d}_1)\frac{s^2}{(s-u)^2}.$$
Therefore, in $U_s$, by \eqref{pfth4.7}, we have
$$\phi\leq K(N,\C, \tilde{d}_1)e^{-\frac{s}{s-u}}\frac{s^2}{(s-u)^2}.$$
Note that for any $t\in[0,s],$
$$\varphi(t)=e^{-\frac{s}{s-t}}\frac{s^2}{(s-t)^2}\leq 4e^{-2}.$$
We obtain that at any point $x\in U_s$,
\be\label{ub1.3}
\phi\leq K(N,\C,\tilde{d}_1).
\ee
Now, for any $x\in\R^n$, we can choose $s>0$ large such that $x\in U_{s/2}$. Then by \eqref{ub1.3} and \eqref{pfth4.7}, we conclude
\[\la_1(x)\leq K(N, \C, \tilde{d}_1, \tilde{d}_2).\]
Since $x$ is arbitrary, we finish proving Proposition \ref{ubprop1}.
\end{proof}
Theorem \ref{theo3} follows from Proposition \ref{es-th1} and Proposition \ref{ubprop1} immediately.

\end{document}